\documentclass[11 pt, twoside]{amsart}
\usepackage{amssymb,latexsym,amsthm,amsmath,mathtools,amsfonts,hyperref,fancyhdr,comment,enumitem, cleveref,xcolor,bbm, ytableau, caption, amssymb, mathrsfs, tikz, subcaption}
\usepackage[a4paper,margin=1in]{geometry}

\hypersetup{
	colorlinks=true,
	linkcolor=blue,
	filecolor=magenta,
	urlcolor=cyan,
	citecolor=red,
}

\newcommand{\C}{\ensuremath{\mathscr{C}}}

\newcommand{\psl}{\textup{PSL}}

\newcommand{\Mod}{\mathrm{mod}}
\newcommand{\supp}{\mathrm{Supp}}

\newcommand{\dcd}{\textup{dcd}}
\newcommand{\gn}{\textup{gn}}
\newcommand{\cn}{\textup{cn}}

\newtheorem{theorem}{Theorem}[section]
\newtheorem{lemma}[theorem]{Lemma}
\newtheorem{corollary}[theorem]{Corollary}

\newtheorem*{theorem*}{Theorem}

\newtheorem{observation}[theorem]{Observation}

\newtheorem{example}[theorem]{Example}

\numberwithin{equation}{section}
\newcommand{\ignore}[1]{}

\newcommand{\mynote}[1]{}

\begin{document}
	\title{On Conjugacy Classes of Derangements in Symmetric and Alternating Groups}

    \author{Harish Kishnani}
    \address{Indian Institute of Science Education and Research Pune, India}
    \email{harishkishnani11@gmail.com}
    
    \author{Rijubrata Kundu}
    \address{Department of Mathematics, Birla Institute of Technology and Science, Pilani, Rajasthan, 333031, India}
    \email{rijubrata8@gmail.com, rijubrata.kundu@pilani.bits-pilani.ac.in}
	
    \subjclass[2020]{20B30, 20B35, 20B05, 20D06, 05A05}
	\date{\today}

    \keywords{Permutation groups, Derangements, Products of conjugacy classes}

    \begin{abstract}
        In this article, we prove two conjectures of Burness and Fusari [Timothy Burness and Marco Fusari, On derangements in simple permutation groups, Forum Math. Sigma 13 (2025)] concerning the powers and products of conjugacy classes of derangements in the symmetric and alternating groups: (1) We show that there exist two conjugacy classes $C$ and $D$ of derangements in $S_n$ such that $S_n=C^2\cup CD$, and (2) We show that there exists a conjugacy class $C$ of derangements in $A_n$ such that $C^2=A_n$, whenever $n\equiv 3\;(\Mod\;4)$. In fact, our result concerning the second conjecture holds in a considerably more general setting, which also answers affirmatively a question posed by Bertram  [Edward Bertram, Even permutations as a product of two conjugate cycles, J. Comb. Theory, Ser. A 12 (1972), 368-380] in a particular case. Moreover, we show that any conjugacy class $C$ of derangements in $S_n$ (resp. $A_n$)  contains a pair of elements that generate $S_n$ or $A_n$ (resp. $A_n$), unless $C$ is the conjugacy class of fixed-point-free involutions.
    \end{abstract}
	
	\maketitle
    
	\section{Introduction}

    Let $G$ be a finite transitive permutation group on $\Omega$, $|\Omega|\geq 2$, with point stabilizer $H\leq G$. Recall that an element of $G$ is called a derangement if it is fixed-point-free on $\Omega$. A classical result of Jordan asserts that a finite transitive group $G$ contains a derangement. Derangements have been widely studied in the context of permutation group theory, leading to several remarkable results and applications in graph theory, number theory, and topology (see \cite{bg} for a detailed study of some aspects of derangements and their applications). In a recent paper (see \cite{bf}), Burness and Fusari studied the set of derangements $\Delta(G)$ of a finite simple transitive group $G$  in several directions, namely, they studied the proportion of derangements in $G$, covering the group $G$ by the product of conjugacy classes of derangements, and generation of $G$ by conjugate derangements. In this article, we consider the symmetric group $S_n$ (resp. alternating group $A_n$) on $n$ letters with its usual action on $n$ letters, that is, with point stabilizer $S_{n-1}$ (resp. $A_{n-1}$), and study some aspects of covering these groups by the product of conjugacy classes of derangements, and generation of these groups by conjugate derangements.

    \medskip

    Recall that the product of conjugacy classes $C_1, C_2, \cdots, C_k$ of a group $G$ is $C_1C_2\cdots C_k:=\{g_1g_2\cdots g_k \mid g_i \in C_i\;\text{for all}\;1\leq i\leq k\}$.  In particular, for a conjugacy class $C$ of $G$, $C^k:=\{g_1g_2\cdots g_k \mid g_i \in C\;\text{for all}\;1\leq i\leq k\}$. In \cite[Theorem 3.8]{bf}, the authors proved that if $G=S_n$ ($n\geq 4$) is a transitive permutation group with pointwise stabilizer $H\leq S_n$ and $H\neq S_{n-1}$, then there exists a pair of conjugacy classes of derangements $C, D$ such that $\displaystyle S_n=C^2\cup CD$. In Remark 3.9, they further conjecture that if $H=S_{n-1}$, setting $C=x^{G}$ and $D=y^G$, where $x=(1,2,\ldots,n)$ and $y=(1,2)(3,4,\ldots,n)$, one may conclude that $S_n=C^2\cup CD$. The first result of this paper proves this in the affirmative, thereby completing their result in full generality. For the convenience of stating the theorem and subsequent results, we introduce some notations at this point. Let $\lambda=(\lambda_1,\cdots, \lambda_r)\vdash n$ be a partition of a natural number $n$, where $\lambda_1\geq \cdots \geq \lambda_r>0$ and $\sum_{i}\lambda_i=n$. Alternatively, in frequency notation, we may write $\lambda=(1^{n_1},\cdots,i^{n_i},\cdots)$ is a partition of $n$ where $\sum_i in_i=n$. Let $l(\lambda)$ denote the length of the partition $\lambda \vdash n$. It is well-known that conjugacy classes of $S_n$ are in one-to-one correspondence with partitions of $n$. Let $\C_{\lambda}$ denote the conjugacy class of $S_n$ indexed by $\lambda$. If $\lambda=(l, 1^{n-l})\vdash n$, we simply write $\C_l$ in place of $\C_{(l,1^{n-l})}$, which is the conjugacy class of $l$-cycles in $S_n$. We have the following:

    \begin{theorem}\label{prod_classes_symmetric}
        Let $n\geq 4$. Then, $S_n=\C_{n}^2\cup \C_{n}\C_{(n-2,2)}$.
    \end{theorem}

    \noindent We prove the above result in Section \ref{proof_first_theorem} using a simple character-theoretic argument. A classical result of Frobenius and the well-known Murnaghan-Nakayama rule for evaluating character values of the complex irreducible representations of $S_n$ are the two main tools used in the proof.

    \medskip

    Let $G=A_n$ be a simple transitive permutation group with point-stabilizer $H\leq A_n$.  It is natural to expect that there exists a conjugacy class $C$ of derangements such that $C^2=A_n$. In \cite[Proposition 3.5]{bf}, the authors show that this holds apart from a few exceptions where they exhibit a conjugacy class of derangements $C$ of $A_n$ such that $A_n=\{1\}\cup C^2$. These exceptions are (a) $(G,H)=(A_5, A_4),\; (A_5,S_3)$ or $(A_8, 2^4:(S_3\times S_3)$, (b) $n\geq  27$, $n\equiv 3\;(\Mod\; 4)$, and $H=A_{n-1}$. In Remark 3.6, they mention that while the exceptions in (a) are genuine exceptions in the sense that there exists no conjugacy class of derangements $C$ such that $C^2=A_n$, they conjecture that (b) is not a genuine exception. More precisely, they conjecture that if $n\geq 27$ and $n\equiv 3\;(\Mod\; 4)$, then $A_n=\C_{(n-4,2,2)}^2$. We note here that $\C_{(n-4,2,2)}$ is also an $A_n$-conjugacy class. Our second theorem proves this conjecture. In fact, we are able to prove a more general version. To state it, we need a few more notations. For $\sigma \in S_n$, let $\supp(\sigma)$ be the set of symbols moved by $\sigma$ and $|\supp(\sigma)|:=m_{\sigma}$. Let $n_i(\sigma)$ denote the number of cycles of length $i$ in the disjoint cycle decomposition (abbrv. dcd) of $\sigma$. Let $n_{\sigma}$ denote the number of non-trivial cycles in the dcd of $\sigma$, that is, $\displaystyle n_{\sigma}=\sum_{i>1}n_{i}(\sigma)$. We have the following:
   \begin{theorem}\label{prod_classes_alternating}
         Let $n\geq 5$, $k\geq1$ and $\lambda=(\ell, 2^k, 1^{n-\ell-2k})\vdash n$. Let $\sigma \in A_n$ be such that $\ell \geq \frac{m_\sigma + n_\sigma}{2}$. Then $\sigma \in \C_\lambda^2$. In particular, if $\ell \geq \lfloor \frac{3n}{4} \rfloor$, then $\C_\lambda^2 = A_n$.
    \end{theorem}

    \noindent It is worthwhile to mention here that our result is a special case of a more general question of Bertram (see concluding remarks of \cite{be}), where he asks that, in order for a conjugacy class $C\subseteq S_n$ satisfy $C^2=A_n$ is it sufficient that a representative permutation in $C$ contains a cycle of length $\geq \lfloor \frac{3n}{4} \rfloor$? Brenner in \cite{br} answered this question in the affirmative, provided $C$ corresponds to a partition with exactly two parts. The study of products and powers of conjugacy classes in finite groups, especially for finite non-abelian simple groups, has become a classical theme in finite group theory, driven by several interesting problems and remarkable conjectures, such as Thompson's conjecture and the Arad-Herzog conjecture. Let $C$ be a conjugacy class of a group $G$. The covering number of $C$ (if it exists) is defined to be the least positive integer $k$ such that $C^k=G$. A result of Arad-Herzog-Stavi states that if $G$ is a finite non-abelian simple group and $\{1\}\neq C$ is a conjugacy class, then the covering number of $C$ exists. In this language, Theorem \ref{prod_classes_alternating} provides a new family of conjugacy classes of $A_n$ (when $l+k\equiv 1\;(\Mod\;2)$) with covering number 2. There are many interesting results on covering numbers of conjugacy classes of finite non-abelian simple groups. We refer the readers to \cite{ah, bh, hgl, kkm, Lev1, Lev2} for several results concerning covering numbers of conjugacy classes of the alternating groups $A_n$ and that of the projective special linear groups $\psl(n,K)$ over a field $K$. 
    \medskip

    Our final theorem concerns the generation of $S_n$ or $A_n$ by conjugacy classes of derangements. In \cite[Theorem G]{bf}, the authors prove that if $G$ is a finite transitive simple group with point stabilizer $H\leq G$, then there exists a pair of conjugate derangements that generate $G$. In the special case of $G=S_n$ (resp. $G=A_n$) with point stabilizer $S_{n-1}$ (resp. $A_{n-1}$), we prove the following result in Section \ref{proof_second_theorem}.

    \begin{theorem}\label{generation_by_conjugate_derangements}
        Let $n\geq 5$. Let $C$ be a conjugacy class of derangements in $S_n$ or $A_n$, and $C$ is not the product of transpositions. Then, for any randomly chosen element $x\in C$, there exists a $y\in C$ such that $x$ and $y$ generate $A_n$ or $S_n$.
    \end{theorem}

    \noindent In fact, in the above theorem, for a randomly chosen element $x\in C$ we produce several candidates $y\in C$ such that $x$ and $y$ generate $A_n$ or $S_n$ (see Corollary \ref{several_candidates}). A group $G$ is called $\frac{3}{2}$-generated if for a randomly chosen element $x\in G$, there exists a $y\in G$ such that $x$ and $y$ generate $G$. Guralnick and Kantor (see \cite{gk}) proved that every finite non-abelian simple group is $\frac{3}{2}$-generated. In general, in a 2-generated group $G$,  for a randomly chosen $x\in G$, there may not exist a conjugate of $x$ that generates $G$. From this viewpoint, Theorem \ref{generation_by_conjugate_derangements} asserts that derangements satisfy $\frac{3}{2}$-conjugate-generation, that is, for any randomly chosen derangement, there exists a conjugate derangement such that they together generate the group.  The study of generating sets of the symmetric and alternating groups is a well-established theme. There is a plethora of interesting results in this direction. In Section \ref{proof_third_theorem}, before proving Theorem \ref{generation_by_conjugate_derangements}, we give a brief account of some of the well-known results concerning generating sets of $S_n$ and $A_n$ consisting of conjugate elements, in order to place our result in a wider context.
    
    \section{Covering the symmetric group by products of derangements}\label{proof_first_theorem}
    In this section, we prove Theorem \ref{prod_classes_symmetric}. The following classical result of Frobenius facilitates the use of character theory in understanding products and powers of conjugacy classes in finite groups.

    \begin{theorem}\label{Product_Character_Theory}
        Let $G$ be a finite group with conjugacy classes $C_i = g_i^G$ for $i = 1, 2$. For an element $x \in G$, let $N(x)$ denote the number of solutions to the equation $x = y_1 y_2$ with $y_i \in C_i.$ Then
        \[N(x) = \frac{|C_1|\,|C_2|}{|G|}  \sum_{\chi \in \mathrm{Irr}(G)}  \frac{\chi(g_1)\chi(g_2)\chi(x)}{\chi(1)},\]
        where $\mathrm{Irr}(G)$ denotes the set of complex irreducible characters of $G$.
    \end{theorem}

    \noindent Thus, for any two conjugacy classes $C_1,C_2$ of $G$ and $x\in G$, $x\in C_1C_2$ if and only if $N(x)>0$. The well-known Murnaghan-Nakayama rule is a recursive formula for evaluating complex irreducible character values of the symmetric group. To state it, we need some notations. For partition $\lambda,\mu \vdash n$, let $\chi_{\lambda}$ denote the irreducible character of $S_n$ indexed by $\lambda \vdash n$ and $w_{\mu}$ be a representative of the conjugacy class $\C_{\mu}$. We recall that $\chi_{(n)}$ is the trivial character and $\chi_{(1^n)}
    $ is the sign character. Thus, $\chi_{(n)}(w_{\mu})=1$ and $\chi_{(1^n)}(w_{\mu})=(-1)^{n-l(\mu)}$ for every $\mu\vdash n$. We also recall the important relation that $\chi_{\lambda'}=\chi_{(1^n)}\chi_{\lambda}$, that is, $\chi_{\lambda'}$ is the twist of $\chi_{\lambda}$ by the sign character, where $\lambda^{'}$ denotes the transpose of $\lambda$. The Young diagram of a partition $\lambda=(\lambda_1,\lambda_2,\ldots)$ is made up of left-justified rows of boxes, where the $i$-th row has $\lambda_i$ boxes. We denote it by $T_{\lambda}$. The $(i,j)$-th cell of $T_{\lambda}$ is the box in the $i$-th row and $j$-th column. For each cell, there is an associated hook, that is, the collection of cells that are to the right of it in the same row and below it in the same column, including the cell itself. The hook length of a cell is the number of boxes in the associated hook. Let $h_{ij}$ denote the hook length of the $(i,j)$-th cell of $T_{\lambda}$. The rim-hook corresponding to the $(i,j)$-th cell is the collection of boxes at the $(k,l)$ positions such that $k\geq i, l\geq j$ and $T_{\lambda}$ does not contain a box in the $(k+1,l+1)$ position. The  rim-hook corresponding to the $(i,j)$-th cell is denoted by $rim_{ij}$. It is easy to see that the number of boxes in $rim_{ij}$ equals $h_{ij}$. Moreover, removing $rim_{ij}$ from $T_{\lambda}$ yields a Young diagram. Finally, let $l_{ij}$ denote the leg-length at the $(i,j)$-th cell, that is, the number of boxes in $j$-th column appearing strictly below the $(i,j)$-th cell. The following example illustrates these notions.

    \begin{example}
        The following is the Young diagram $T_{\lambda}$ of the partition $\lambda=(7,5,3,3,1)$ with its hook lengths:
        \begin{figure}[ht]
        \centering
        \ytableausetup{centertableaux, boxsize=1.2em}
        \begin{ytableau}
        11 & 9 & 8 & 5 & 4 & 2 & 1 \\
        8  & 6  & 5 & 2 & 1 \\
        5  & 3  & 2 \\
        4  & 2  & 1 \\
        1
        \end{ytableau}
        \caption{Young diagram of the partition $(7,5,3,3,1)$ with hook lengths.}
        \label{fig:young-hooklengths}
        \end{figure}
        
        \noindent In the next diagram, the red-coloured boxes and blue-coloured boxes constitute the rim-hook corresponding to the  $(3,1)$-cell and $(2,2)$-cell, respectively.
        
        \begin{figure}[ht]
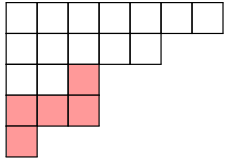
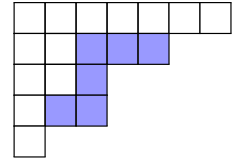

        \centering
        \ytableausetup{centertableaux, boxsize=1 em}

        \begin{subfigure}{0.45\textwidth}
        \centering
        \begin{ytableau}
        *(white) & *(white) & *(white) & *(white) & *(white) & *(white) & *(white) \\
        *(white) & *(white) & *(white) & *(white) & *(white) \\
         & & *(red!40) \\
         *(red!40) & *(red!40) & *(red!40) \\
        *(red!40)
        \end{ytableau}
        \caption{$rim_{3,1}$ highlighted in red.}
        \end{subfigure}
        \hfill
        \begin{subfigure}{0.45\textwidth}
        \centering
        \begin{ytableau}
        *(white) & *(white) & *(white) & *(white) & *(white) & *(white) & *(white) \\
        *(white) &  & *(blue!40) & *(blue!40) & *(blue!40) \\
        *(white) &  & *(blue!40) \\
        *(white) & *(blue!40) & *(blue!40) \\
        *(white)
        \end{ytableau}
        \caption{$rim_{2,2}$ highlighted in blue.}
        \end{subfigure}

        \caption{Young diagrams of the partition $(7,5,3,3,1)$ with two different rim-hooks highlighted.}
        \label{fig:young-borderstrips}
    \end{figure}

    \noindent Thus, after removing $rim_{3,1}$  from  $T_{\lambda}$ we are left with the Young diagram of $(7,5,2)$. Similarly, after removing  $rim_{2,2}$ from $T_{\lambda}$ we are left with the Young diagram of $(7,2,2,1,1)$. Finally, we have $l_{3,1}=l_{2,2}=2$.
    \end{example}

    \noindent Now we can state the recursive Murnaghan-Nakayama rule. 
    
    \begin{theorem}\cite[Theorem 5.11.1]{ap}\label{MN_rule}
        Let $\lambda,\mu\vdash n$ and $\mu=(\mu_1,\mu_2,\ldots,\mu_k)$. Then,
        \[
        \chi_{\lambda}(w_{\mu})=\sum_{h_{ij}=\mu_k}(-1)^{l_{ij}}\chi_{\lambda\setminus rim_{ij}}(w_{\hat{\mu}})
        \]
        where $\hat{\mu}=(\mu_1,\mu_2,\ldots,\mu_{k-1})$.
    \end{theorem}

    \begin{proof}[Proof of Theorem \ref{prod_classes_symmetric}]
    
    By a result of Bertram given in the next section (see Theorem \ref{Bertram}), we have  $\C_n^2=A_n$. Thus, to prove the above theorem, it only remains to show that $\C_n\C_{(n-2,2)}=S_n\setminus A_n$. Let $\pi \in S_n\setminus A_n$. To show $\pi \in \C_n\C_{(n-2,2)}$, it is enough to show that $\displaystyle N(\pi)= \sum_{\lambda \vdash n}\frac{\chi_{\lambda}(w_{(n)})\chi_{\lambda}(w_{(n-2,2)})\chi_{\lambda}(\pi)}{\chi_{\lambda}(1)}>0$. Using Theorem \ref{MN_rule}, it is easy to see that $\chi_{\lambda}(w_{(n)})=0$, unless $\lambda=(n-k,1^k)$, in which case $\chi_{\lambda}(w_{(n)})=(-1)^k$. Thus the above sum vanishes for all $\lambda\neq (n-k,1^k)$ and we get $\displaystyle N(\pi)=\sum_{k=0}^{n-1}\frac{(-1)^k\chi_{(n-k,1^k)}(w_{(n-2,2)})\chi_{(n-k,1^k)}(\pi)}{\chi_{(n-k,1^k)}(1)}$. Now $\chi_{(n-k,1^k)}(w_{(n-2,2)})\neq 0$ implies that $T_{(n-k,1^k)}$  has a hook of length $n-2$, which is possible only when $0\leq k\leq 1$ or $n-2\leq k\leq n-1$. We recall that $\chi_{(n-1,1)}(\pi)=n_1(\pi)-1$, and hence $\chi_{(n-1,1)}(w_{(n-2,2)})=-1$. This  yields 
     
     $$\displaystyle N(\pi)=2+2\frac{n_1(\pi)-1}{n-1}>0$$
     
     \noindent Our proof is now complete.
    \end{proof}

    \section{Covering alternating groups by squares of conjugacy classes of derangements}\label{proof_second_theorem}

    Before stepping into the proof of Theorem \ref{prod_classes_alternating}, we mention a few things about the covering of $A_n$ by the squares of conjugacy classes of derangements in $A_n$. Bertram proved the following result.
	\begin{theorem}\cite{be}\label{Bertram}
		Let $2\leq l\leq n$ and $\C_l$ denote the conjugacy class of $l$-cycles in $S_n$. If $n\neq 4$, then $\C_l^2=A_n$ if and only if $\lfloor \frac{3n}{4} \rfloor\leq l\leq n$. When $n=4$, $\C_l^2=A_4$ if and only if $2\leq l\leq 4$.
	\end{theorem}

     We recall that a $S_n$-conjugacy class $\C_{\lambda}\subseteq A_n$ splits into two $A_n$-classes if and only if $\lambda$ has distinct and odd parts; otherwise, $\C_{\lambda}$ is an $A_n$-conjugacy class. When $\lambda$  has distinct and odd parts, let $\C_{\lambda}^{+}$ and $\C_{\lambda}^{-}$ be the  two split  $A_n$-classes. If $n$ is odd, the following theorem determines the covering number of the two $A_n$-classes $\C_n^{\pm}$.

    \begin{theorem}\cite[Theorem 1]{lt}\label{larsen-tiep}
        Let $2\nmid n\ge 7$. Then $\C_{n}^{\pm}\C_{n}^{\pm}\supseteq A_n\setminus\{1\}$.
    \end{theorem}

    \noindent From the above theorem, $\C_{n}^{\pm 2}=A_n$, whenever $n\geq 9$ and $n\equiv 1\;(\Mod\;4)$. The following is a result of Brenner (see \cite{br}).

    \begin{theorem}\label{Brenner_two_orbits}
        Let $n\geq 5, \ell_1, \ell_2 \geq 2$, $\ell_1 + \ell_2 > \lfloor \tfrac{3}{4}n \rfloor + 3$, and $\lambda=(\ell_1,\ell_2)\vdash n$. Then $A_n = \C_{\lambda}^2$.
    \end{theorem}

    \noindent Note that $\C_{(n-2,2)}$ is a conjugacy class of $A_n$ whenever $n$ is even. The previous theorem shows that $\C_{(n-2,2)}^2=A_n$ whenever $n$ is even and $n\geq 14$. For $n\in \{6,8,10,12\}$, $\C_{(n-2,2)}^2=A_n$ holds as well and can be verified by \textsf{GAP}. Thus, combining these cases, we can conclude that there exists a conjugacy class of derangements $C$ of $A_n$ such that $C^2=A_n$, whenever $n\not\equiv 3\;(\Mod\;4)$. From Theorem \ref{prod_classes_alternating}, we conclude that $\C_{(n-4,2^2)}^2=A_n$, whenever $n$ is odd and $n\geq 15$, thereby proving the conjecture of Burness-Fusari. When $n\in \{7,11\}$, once again a check using \textsf{GAP} yields that $\C_{(n-4,2,2)}^2=A_n$. Overall, there exists a conjugacy class $C$ of derangements in $A_n$ such that $C^2 = A_n$, whenever $n\geq 6$. The proof of Theorem \ref{prod_classes_alternating} is constructive, and we use the ideas of Bertram's constructive proof of Theorem \ref{Bertram}. For that reason, we briefly mention the key ideas of Bertram's proof.

    \subsection{Bertram's construction} Let $l\geq 2$. We outline the proof of the following more general theorem of \cite{be} which yields Theorem \ref{Bertram}.

    \begin{theorem}\label{Bertram_intermediate}
     Let $n\geq 5$ and $\sigma \in A_n$. Then $\sigma$ can be written as a product of two $l$-cycles where $\dfrac{m_{\sigma}+n_{\sigma}}{2}\leq l\leq n$.
    \end{theorem}

    First we outline the proof for $l=\dfrac{m_{\sigma}+n_{\sigma}}{2}$. Since $\sigma\in A_n$, the dcd of $\sigma$ consists of odd-length cycles and an even number of even-length cycles. Keeping this in mind we do the construction into two steps. First, we do the construction for the two basic fragments: (1) a single odd-length cycle, (2) the product of two even-length cycles. In the second step, we show how to glue these pieces together to get the construction for a general element.

    \medskip

    \noindent \textbf{Step 1:} Writing the basic fragments into product of two $l$-cycles.

    \begin{enumerate}
        \item Let $\sigma=(x_1,x_2,\dots,x_{2k+1})$, where $k\geq 1$, that is, $\sigma$ is a odd-length cycle of length $\ge 3$. Then,
        $$\sigma=(x_{k+2},x_{k+3},\ldots, x_{2k+1},x_1)(x_1,x_2,\ldots,x_{k+1}).$$
        For example, $(1\;2\;3\;4\;5\;6\;7\;8\;9)=(6\;7\;8\;9\;1)(1\;2\;3\;4\;5)$.

        \item  Let $\sigma=(x_1,x_2,\dots,x_{2r})(y_1,y_2,\dots,y_{2s})$. Here we have to make three subcases.

       \noindent \textbf{Subcase-I}: Assume that $r,s\geq 2$, that is, $\sigma$ is a disjoint product of two even-length cycles, both of length $\geq 4$. Then,
       $$\sigma=(y_{s+1},\ldots,y_{2s},y_1,x_2,\ldots,x_{r+1})(x_{r+1},x_{r+2},\ldots,x_{2r},x_1,y_1,y_2,
       \ldots,y_s).$$
       For example, $(1,2,3,4,5,6)(7,8,9,10)=(9,10,7,2,3,4)(4,5,6,1,7,8)$.

       \noindent \textbf{Subcase-II}: Let $r=s=1$, that is, $\sigma$ is a disjoint product of two transpositions. Then,
       $$\sigma=(y_2,y_1,x_2)(x_2,x_1,y_1).$$
       For example, $(1,2)(3,4)=(4,3,2)(2,1,3)$.

       \noindent \textbf{Subcase-III}: Assume $r\geq 2$ and $s=1$, that is, $\sigma$ is a disjoint product of two even-length cycles, one being of length $\geq 4$ and the other being a transposition. Then,
       $$\sigma=(y_2,y_1,x_2,\ldots,x_{r+1})(x_{r+1},x_{r+2},\ldots,x_{2r},x_1,y_1).$$
       For example, $(1,2,3,4)(5,6)=(6,5,2,3)(3,4,1,5)$.
    \end{enumerate}
    We notice that if $\sigma=\pi\tau$, where $\pi$ and $\tau$ are obtained as above, then $l(\pi)=l(\tau)=\dfrac{m_{\sigma}+n_{\sigma}}{2}$, as required. Now we need to glue these ``basic fragments'' together, which is illustrated in the next step.

    \medskip

    \noindent \textbf{Step 2}: Let $\sigma=\sigma_1\sigma_2\cdots\sigma_k$ where $\sigma_i$ are the ``basic fragments'' of the above step (which are obtained from the dcd of $\sigma$). Let $\sigma_i=\pi_{i}\tau_{i}$, where $\pi_{i}$ and $\tau_{i}$ are both cycles of length $l_i=\dfrac{m_{\sigma_i}+n_{\sigma_i}}{2}$. Then, $\sigma=\pi\tau$, where $\pi=(\pi_{1},\pi_{k},\cdots,\pi_{2})$ and $\tau=(\tau_{1},\tau_{2},\cdots, \tau_{k})$, where $(\pi_{1},\pi_{k},\cdots,\pi_{2})$ is simply concatenating the symbols of $\pi_{1}$, $\pi_{k}$, $\cdots$, $\pi_{2}$ and the same applies to $\tau$. We illustrate this via several examples.

    \noindent If $\sigma=(1,2,3,4,5)(6,7,8,9,10,11,12)$, then $\sigma_1=(1,2,3,4,5)$ and $\sigma_2=\allowbreak(6,7,8,9,10,11,12)$. $\sigma_1=\pi_{1}\tau_{1}=(4,5,1)(1,2,3)$ and $\sigma_2=\pi_{2}\tau_{2}=(10,11,12,6)(6,7,8,9)$. Then, 
    $$\sigma=\pi\tau=(4,5,1,10,11,12,6)(1,2,3,6,7,8,9).$$

    \noindent If $\sigma=(1,2,3,4,5)(6,7,8,9)(10,11,12,13,14,15)$, then $\sigma_1=(1,2,3,4,5)$ and $\sigma_2=(6,7,8,9)\allowbreak(10,11,12,13,14,15)$. So, $\sigma_1=\pi_{1}\tau_{1}=(4,5,1)(1,2,3)$ and $\sigma_2=\pi_{2}\tau_{2}=(13,14,15,10,7,8)\allowbreak(8,9,6,10,11,12)$. Then, 
    $$\sigma=\pi\tau=(4,5,1,13,14,15,10,7,8)(1,2,3,8,9,6,10,11,12).$$

    \noindent If $\sigma=(1,2,3,4,5)(6,7,8,9)(10,11)(12,13)(14,15)$, then $\sigma_1=(1,2,3,4,5)$ and $\sigma_2=(6,7,8,9)\allowbreak(10,11)$ and $\sigma_3=(12,13)(14,15)$. $\sigma_1=\pi_{1}\tau_{1}=(4,5,1)(1,2,3)$, $\sigma_2=\pi_{2}\tau_{2}=(11,10,7,8)\allowbreak(8,9,6,10)$, and $\sigma_3=\pi_{3}\tau_{3}=(15,14,13)(13,12,14)$. Then, 
    $$\sigma=\pi\tau=(4,5,1,15,14,13,11,10,7,8)(1,2,3,8,9,6,10,13,12,14).$$

    \noindent Thus, the above construction allows us to write $\sigma$ as a product of two $l$-cycles where $l=\dfrac{m_{\sigma}+n_{\sigma}}{2}$. 

    \medskip

    Assume now that $\sigma\in A_n$ and $\sigma=\pi\tau$ where both $\pi$ and $\tau$ are cycles of length $\dfrac{m_{\sigma}+n_{\sigma}}{2}$. Now we can increase the length of $\pi$ and $\tau$ one at a time without changing the product $\pi\tau=\sigma$. We can do this till both $\pi$ and $\tau$ have length $n$. We explain this process of lengthening of cycles in brief. We need the following observation.
    \begin{observation}
        In the Bertram's construction, since we need $\ell \geq \frac{m_\sigma+ n_\sigma}{2}$, we always have $2\ell \geq m_\sigma +n_\sigma > m_\sigma$. So, if $\sigma = \pi \tau$, where $\pi$ and $\tau$ are both cycles of equal length obtained using the Bertram's construction, then $\pi$ and $\tau$ must have at least one common letter.
    \end{observation}

    \noindent Suppose $\pi=(a_1,a_ 2,\ldots, a_r)$ and $\tau=(b_1,b_2,\ldots, b_r)$. Suppose $a_i\notin \supp(\tau)$ and $b_j\notin \supp(\pi)$, then define $\pi'=(a_1,\ldots,a_i,b_j,\ldots,a_r)$ and $\tau'=(b_1,\ldots,a_i,b_j,\ldots,b_r)$. Then, it is clear that $\pi\tau=\pi'\tau'$. Now assume that there exists a letter $c$ such that $c\notin \supp(\pi)\cup \supp(\tau)$. Then, by the observation made above, there exists $i$ and $j$ such that $a_i=b_j$. Define $\pi'=(a_1,\ldots,a_i,c,\ldots,a_r)$ and $\tau'=(b_1,\ldots,c,b_j,\ldots,b_r)$. Once again it is clear that $\pi\tau=\pi'\tau'$. We give an example to illustrate this lengthening process.

    \noindent Let $(1,2,3,4,5)\in A_{8}$. Then, $\sigma=\pi\tau$, where $\pi=(4,5,1)$ and $\tau=(1,2,3)$. Now,
    \begin{align*}
        (1,2,3,4,5)=\pi\tau&=(4,5,2,1)(1,5,2,3)=(4,3,5,2,1)(1,5,2,4,3)\\
        &=(4,3,5,2,6,1)(1,5,6,2,4,3)\\
        &=(4,7,3,5,2,6,1)(1,5,6,2,7,4,3)\\
        &=(4,7,3,8,5,2,6,1)(1,5,6,2,7,4,8,3)
    \end{align*}
    From the construction, it is clear that there are many choices in which we may proceed to lengthen the cycles. Now we are ready to prove Theorem \ref{prod_classes_alternating}. We prove it through a series of lemmas. We introduce a notation for our convenience. Let $\sigma\in S_n$. Denote by $\dcd^*(\sigma)$ the set of non-trivial cycles occuring (that is, cycles of length greater than 1) in the disjoint cycle decomposition of $\sigma$.
    
    \begin{lemma}\label{lemma: large odd cycle}
        Let $n\geq 5$, $k\geq1$ and $\lambda = (\ell, 2^k, 1^{n-\ell-2k})$. Let $\sigma \in A_n$ be such that $\ell \geq \frac{m_\sigma + n_\sigma}{2}$. If $\ell +2k-1 \leq m_\sigma \leq n$ and there exists a cycle of odd length $t \geq 2k +3$ in $\dcd^*(\sigma)$, then $\sigma \in \C_\lambda^2$.
    \end{lemma}

    \begin{proof}
        Without loss of generality, let the cycle of odd length $t \geq 2k +3$ in $\dcd^*(\sigma)$ be $\sigma_1=(1,2,\dots, t)$. We first write $\sigma_1$ as $\sigma = \rho_1 \rho_2$, where  $\rho_1 = \left(\frac{t+3}{2}, \frac{t+5}{2}, \cdots, t-k, k+2,1\right)(2,t)(3,t-1) \cdots  (k+1, t-k+1)$, $\rho_2 = \left(1,t,k+2,k+3,\cdots,\frac{t+1}{2}\right)(2,t-1)(3,t-2) \cdots  (k+1, t-k).$ Observe that $\rho_1$ and $\rho_2$ have $k$-many $2$-cycles and one cycle of length $\ell_1:=\frac{t-2k+3}{2}$.

        \medskip

        We first assume that $\ell +2k \leq m_\sigma \leq n$. With this assumption, if $\sigma \in \C_\lambda^2$ in $A_{m_\sigma}$, then $\sigma \in \C_\lambda^2$ in $A_{n}$, for any $n \geq m_\sigma$. Thus, it is enough to assume that $\ell +2k \leq m_\sigma=n$. Now, the permutation $\sigma\sigma_1^{-1}$ has $m_{\sigma\sigma_1^{-1}}= m_\sigma-t$ and $n_{\sigma\sigma_1^{-1}}= n_\sigma -1$. So, $\sigma\sigma_1^{-1}$ can be written as a product of two cycles of length $\ell_2:=\frac{m_\sigma+n_\sigma-t-1}{2}$, using Bertram's construction. Let these cycles be $\pi_1 = (a_1,a_2,\dots, a_{\ell_2-1},a)$ and $\pi_2 =(a,b_2,b_3,\dots, b_{\ell_2})$, where $a$ is any symbol common in both $\pi_1$ and $\pi_2$. Then we write $\sigma$ as a product of $\rho$ and $\tau$, where $\rho= \left(a_1,\cdots,a_{\ell_2-1},a,\frac{t+3}{2}, \frac{t+5}{2}, \cdots, t-k, k+2,1\right)(2,t)(3,t-1) \cdots  (k+1, t-k+1)$, and $\tau = \left(1,t,k+2,k+3,\cdots,\frac{t+1}{2},a,b_2,\cdots,b_{l_2}\right)(2,t-1)(3,t-2) \cdots  (k+1, t-k)$. Thus $\sigma \in \C_{\lambda_1}^2$, whenever $\lambda_1 = (\ell', 2^k, 1^{n-\ell'-2k})$, $\ell':= \ell_1+\ell_2=\frac{m_\sigma+n_\sigma-2k+2}{2}$. Further, the length $\ell'$ can be increased to $n-2k$ by applying the lengthening process described above. Hence, $\sigma \in \C_{\lambda_1}^2$, whenever $\lambda_1 = (\ell', 2^k, 1^{n-\ell'-2k})$, $\ell'\geq \frac{m_\sigma+n_\sigma-2k+2}{2}$. Since $k\geq 1$, $\frac{m_\sigma+n_\sigma-2k+2}{2} \leq \frac{m_\sigma+n_\sigma}{2}$, and hence $\sigma \in \C_{\lambda}^2$, whenever $\lambda= (\ell, 2^k, 1^{n-\ell-2k})$, $\ell\geq \frac{m_\sigma+n_\sigma}{2}$.

        \medskip

        We now assume that $m_\sigma = \ell +2k-1$. For that, it is enough to assume that $\ell +2k =n$ and $m_\sigma = n-1$. In which case, $\sigma$ has a fixed point in $A_n$, let it be $\alpha$. We first write $\sigma\sigma_1^{-1}$ as a product of two cycles of length $\ell_1:=\frac{m_\sigma+n_\sigma-t-1}{2}$, using Bertram's construction. Then we add the point $\alpha$ in these two cycles, similar to Bertram's construction so that the product fixes $\alpha$. Now, the new length of these cycles is $\frac{m_\sigma+n_\sigma-t-1}{2} +1 = \frac{m_\sigma+n_\sigma-t+1}{2}$. We then proceed as above and get the expression of $\sigma$, showing that $\sigma \in \C_{\lambda_1}^2$, whenever $\lambda_1 = (\ell', 2^k, 1^{n-\ell'-2k})$, $\ell'\geq \frac{m_\sigma+n_\sigma-2k+4}{2}$. If $k\geq 2$, $\frac{m_\sigma+n_\sigma-2k+4}{2} \leq \frac{m_\sigma+n_\sigma}{2}$, and hence $\sigma \in \C_{\lambda}^2$, whenever $\lambda= (\ell, 2^k, 1^{n-\ell-2k})$, $\ell\geq \frac{m_\sigma+n_\sigma}{2}$. Thus, we can now assume that $k=1$, $n=\ell+2$, and $m_\sigma =n-1$. Additionally, for the given $\lambda = (\ell, 2^k, 1^{n-\ell-2k})$, if $\ell$ is such that $\ell-1 \geq \frac{m_\sigma + n_\sigma}{2}$, then we let $\lambda_0 = (\ell_0, 2, 1^{n-\ell_0-2})$, where $\ell_0 =\ell -1$. Hence, $\ell_0+2 = m_\sigma = n-1$. We proceed similar to the previous case ($\ell +2k \leq m_\sigma \leq n$) and conclude that $\sigma \in \C_{\lambda_0}^2$ in $A_{n-1}$, where $\lambda_0= (\ell_0, 2^k, 1^{n-1-\ell_0-2k})$, $\ell\geq \frac{m_\sigma+n_\sigma}{2}$. Let $\sigma = \pi\tau$, where $\pi,\tau \in \C_{\lambda_0}$. Let $\pi_0$ and $\tau_0$ be the cycles of length $\ell_0=\ell-1$ in permutations $\pi$ and $\tau$, respectively. We now add the point $\alpha$ in $\pi_0$ and $\tau_0$ such that their length increases to $\ell$. Let $\pi_1$ and $\tau_1$ be the cycles obtained after increasing the length of $\pi_0$ and $\tau_0$, respectively. Consequently, $\sigma = \pi\tau= \pi_1\pi_0^{-1}\pi\tau_1\tau_0^{-1}\tau \in \C_{\lambda}^2$. Finally, we are left with the assumptions $n=\ell+2$, $m_\sigma =n-1$, and $\lambda = (\ell, 2^k, 1^{n-\ell-2k})$ are such that $\ell \geq \frac{m_\sigma + n_\sigma}{2} > \ell-1 \implies \ell = \frac{m_\sigma + n_\sigma}{2} \implies n-2 = \frac{n-1 + n_\sigma}{2} \implies n_\sigma = n-3$. We always have $m_\sigma \geq 2n_\sigma \implies n-1 \geq 2(n-3) =2n-6 \implies n \leq 5$. Thus, we are left to verify this case only for $\sigma \in A_5$ with $m_\sigma =4$, $n_\sigma =2$, $\lambda = (3,2)$, but here $\sigma$ is a product of two $2$-cycles which contains no cycle of length $2k+3 =5$.
    \end{proof}

    \begin{lemma}\label{lemma: large even cycle}
        Let $n\geq 5$, $k\geq1$ and $\lambda = (\ell, 2^k, 1^{n-\ell-2k})$. Let $\sigma \in A_n$ be such that $\ell \geq \frac{m_\sigma + n_\sigma}{2}$. If $\ell +2k-1 \leq m_\sigma \leq n$ and there exists a cycle of even length $t \geq 2k +4$ in $\dcd^*(\sigma)$, then $\sigma \in \C_\lambda^2$.
    \end{lemma}

    \begin{proof}
        Without loss of generality, let this cycle be $\sigma_1=(1,2,\dots, t)$. Let $\sigma_2:=(a_1, a_2, \dots, a_{2m})$ be some other cycle of even length in $\dcd^*(\sigma)$. We first write $\sigma_1\sigma_2$ as a product of two permutations $\rho_1$ and $\rho_2$, where $\rho_1 = \left(a_{m+2}, a_{m+3}, \dots, a_{2m},a_1,\frac{t+2}{2}, \frac{t+4}{2}, \cdots, t-k, k+2,1\right)(2,t)\cdots  (k+1, t-k+1)$, $\rho_2 =\left(1,t,k+2,k+3,\cdots,\frac{t}{2},a_1,a_2,\dots, a_{m+1}\right)(2,t-1) \cdots  (k+1, t-k)$. Here, $\rho_1$ and $\rho_2$ have $k$-many $2$-cycles and one cycle of length $\ell_1:=\frac{t+2m-2k+4}{2}$.

        \medskip

        We first assume that $\ell +2k \leq m_\sigma \leq n$. Again, it is enough to assume that $m_\sigma=n$. Now, the permutation $\sigma {\sigma_1}^{-1} {\sigma_2}^{-1}$ has $m_{\sigma\sigma_1^{-1}{\sigma_2}^{-1}}= m_\sigma-t-2m$ and $n_{\sigma\sigma_1^{-1}{\sigma_2}^{-1}}= n_\sigma -2$. So, $\sigma\sigma_1^{-1}{\sigma_2}^{-1}$ can be written as a product of two cycles of length $\ell_2:=\frac{m_\sigma+n_\sigma-t-2m-2}{2}$, using the Bertram's construction. We then proceed similarly to Lemma \ref{lemma: large odd cycle} and show that $\sigma \in \C_{\lambda_1}^2$, whenever $\lambda_1 = (\ell', 2^k, 1^{n-\ell'-2k})$, $\ell'\geq \ell_1 +\ell_2 = \frac{m_\sigma+n_\sigma-2k+2}{2}$. Since $k\geq 1$, $\sigma \in \C_{\lambda}^2$, whenever $\lambda= (\ell, 2^k, 1^{n-\ell-2k})$, $\ell\geq \frac{m_\sigma+n_\sigma}{2}$. For $m_\sigma = \ell +2k-1$, we again proceed similarly to Lemma \ref{lemma: large odd cycle}, replacing $\sigma\sigma_1^{-1}$ by $\sigma\sigma_1^{-1}{\sigma_2}^{-1}$.
    \end{proof}

    \begin{lemma}\label{lemma: all cycles are small}
        Let $n\geq 5$, $k\geq1$ and $\lambda = (\ell, 2^k, 1^{n-\ell-2k})$. Let $\sigma \in A_n$ be such that $\ell \geq \frac{m_\sigma + n_\sigma}{2}$. If $\ell +2k-1 \leq m_\sigma \leq n$ and all cycles in $\dcd^*(\sigma)$ have length at most $2k +2$, then $\sigma \in \C_\lambda^2$.
    \end{lemma}

    \begin{proof}
        Let $\sigma_1\sigma_2\dots\sigma_{n_\sigma}$ be the disjoint cycle decomposition of $\sigma$. Recall that $n_i(\sigma)$ denote the number of cycles of length $i$ in $\dcd^*(\sigma)$. If $\sigma$ is understood, we simply write $n_i$ in place of $n_i(\sigma)$. We first reduce our proof to the case $n_2 \leq 1$, and for $\ell_i >2$ either $n_{\ell_i} \leq 1$ or $0<k \leq \ell_i-1$. Let $(a_1,a_2)(a_3,a_4)$ be a pair of $2$-cycles in $\dcd^*(\sigma)$. Then it can be written as $(a_1,a_2)(a_3,a_4) = (a_2,a_4)(a_1,a_3). (a_1,a_4)(a_2,a_3) \in \C_{(2,2)}^2$. Similarly, for any even $r$, if $\sigma_1,\sigma_2, \dots, \sigma_r$ are $2$-cycles in $\dcd^*(\sigma)$, then we can write $\sigma_1\sigma_2 \dots\sigma_r$ as a product of $r$-many $2$-cycles. Consequently, $\sigma \in \C_\lambda^2$ if $\sigma':=\sigma \sigma_1^{-1}\sigma_2^{-1} \dots \sigma_r^{-1} \in \C_{\lambda_1}^2$, where $\lambda_1 = (\ell, 2^{k-r}, 1^{n-\ell-2(k-r)})$. We can choose $r$ such that either $n_2(\sigma') \leq1$, or the number of unused $2$-cycles in $\lambda$, that is $k-r$ is $\leq 1$. Note that if $\ell \geq \frac{m_\sigma + n_\sigma}{2}$, then $\ell \geq \frac{m_{\sigma'} + n_{\sigma'}}{2} = \frac{m_\sigma + n_\sigma -3r}{2}$. Thus, for our purpose, it is enough to assume that either $n_2 \leq 1$, or $k \leq 1$. 

        \medskip

        We first proceed with the assumption that $k \leq 1$. If $k=0$, then $\C_{\lambda}=\C_{\ell}$, the conjugacy class of $\ell$-cycles, and hence $\sigma \in \C_\lambda^2$, by Theorem \ref{Bertram_intermediate}. For $k=1$, we first let $n_2 \geq 2$. Let $\sigma_1:= (1,2)$ and $\sigma_2:=(3,4)$ be a pair of $2$-cycles in $\dcd^*(\sigma)$. Then $\sigma_1 \sigma_2 = (2,4)(1,3). (4,1)(2,3)$. For the permutation, $\sigma\sigma_1^{-1}\sigma_2^{-1}$, $\frac{m_{\sigma\sigma_1^{-1}\sigma_2^{-1}} + n_{\sigma\sigma_1^{-1}\sigma_2^{-1}}}{2} = \frac{m_\sigma + n_\sigma -6}{2}$. Then by Bertram's construction, we write $\sigma\sigma_1^{-1}\sigma_2^{-1}$ as $\sigma\sigma_1^{-1}\sigma_2^{-1} = \rho_1\rho_2 \in \C_{\ell'}^2$, where $\ell' = \frac{m_\sigma + n_\sigma -6}{2}$ or $\ell' = \frac{m_\sigma + n_\sigma -6}{2} + 1$, depending on whether $m_\sigma \geq \ell + 2k$ or $m_\sigma = \ell + 2k-1$. Let $\rho_1 = (a_1, a_2, \dots, a_{\ell'-1}, a)$ and $\rho_2 = (a, b_2, b_3, \dots, b_{\ell'})$. Then $\sigma = (a_1, a_2, \dots, a_{\ell'-1}, a, 2,4) (1,3). (4,1,a, b_2, b_3, \dots, b_{\ell'})(2,3) \in \C_{\lambda'}^2$, where $\lambda' = (\ell' +2, 2, 1^{n-\ell'-4})$. Since $\ell' +2 \leq \frac{m_\sigma + n_\sigma }{2}$, $\sigma \in \C_{\lambda}^2$. Thus, we can assume $n_2 \leq 1$. 

        \medskip

        Now, let $\ell_i>2$ be such that $k \geq \ell_i$. If $n_{\ell_i} \geq 2$, then let $\sigma_1:= (1,2, \dots, \ell_i)$ and $\sigma_2:(\ell_i+1, \ell_i+2, 2\ell_i)$ be two $\ell_i$-cycles in the $\dcd^*(\sigma)$. We write $\sigma_1\sigma_2 = \rho \tau \in \C_{(2^{\ell_i})}^2$, where $\rho= (2,2\ell_i) (3,2\ell_i-1) \cdots (\ell_i, \ell_i+2)(1, \ell_i+1)$ and $\tau=(1,2\ell_i) (2,2\ell_i-1) \cdots (\ell_i -1, \ell_i+2)(\ell_i, \ell_i+1)$. For the permutation $\sigma\sigma_1^{-1}\sigma_2^{-1}$, we have $\frac{m_{\sigma\sigma_1^{-1}\sigma_2^{-1}} + n_{\sigma\sigma_1^{-1}\sigma_2^{-1}}}{2} = \frac{m_\sigma + n_\sigma -2(\ell_i+1)}{2}$. Then by Bertram's construction, we write $\sigma\sigma_1^{-1}\sigma_2^{-1}$ as $\sigma\sigma_1^{-1}\sigma_2^{-1} = \rho_1\rho_2 \in \C_{\ell'}^2$, where $\ell' = \frac{m_\sigma + n_\sigma -2(\ell_i+1)}{2}$ or $\ell' = \frac{m_\sigma + n_\sigma -2(\ell_i+1)}{2} + 1$, depending on whether $m_\sigma \geq \ell + 2k$ or $m_\sigma = \ell + 2k-1$. Now, by a suitable concatenation of symbols, $\sigma \in \C_\lambda^2$. Hence, we can assume that $n_2 \leq 1$, and for $\ell_i >2$ either $n_{\ell_i} \leq 1$ or $k \leq \ell_i-1$. For $k=0$, we already have Bertram's result. Thus, we assume that $n_2 \leq 1$, and for $\ell_i >2$ either $n_{\ell_i} \leq 1$ or $0<k \leq \ell_i-1$. 

        \medskip

        We now focus on $\ell_i =3$ and $\ell_i =4$. For $\ell_i =3$, we have either $n_{3} \leq 1$ or $0<k \leq 2$. If $0<k \leq 2$ and $n_{3} \geq 2$, then let $\sigma_1 := (1,2,3)$ and $\sigma_2:= (4,5,6)$ be a pair of $3$-cycles in $\dcd^*(\sigma)$. We first use the Bertram's construction and write $\sigma\sigma_1^{-1}\sigma_2^{-1}$ as $\sigma\sigma_1^{-1}\sigma_2^{-1} = \rho_1\rho_2 \in \C_{\ell'}^2$, where $\ell' = \frac{m_\sigma + n_\sigma -8}{2}$ or $\ell' = \frac{m_\sigma + n_\sigma -8}{2} + 1$, depending on whether $m_\sigma \geq \ell + 2k$ or $m_\sigma = \ell + 2k-1$. Let $\rho_1 = (a_1, a_2, \dots, a_{\ell'-1}, a)$ and $\rho_2 = (a, b_2, b_3, \dots, b_{\ell'})$. Then for $k=1$, we write $\sigma = \rho\tau$, where $\rho=(a_1, a_2, \dots, a_{\ell'-1}, a, 5,1) (2,4) (3,6)$ and $\tau = (1,4,a, b_2, \dots, b_{\ell'}) (2,6)(3,5)$. For $k=2$, we write $\sigma = \rho\tau$, where $\rho=(a_1, \dots, a_{\ell'-1}, a, 3, 1, 4) (2,6)$ and $\tau =(4,5,2,a, b_2, b_3, \dots, b_{\ell'}) (1,6)$. Thus, $\sigma \in \C_{\lambda}^2$ for $0<k \leq 2$ and $n_{3} \geq 2$. Hence, we can assume $n_{3} \leq 1$. 

        \medskip

        Now, for $\ell_i =4$, we have either $n_{4} \leq 1$ or $0<k \leq 3$. If $0<k \leq 3$ and $n_{4} \geq 2$, then let $\sigma_1 := (1,2,3,4)$ and $\sigma_2:= (5,6,7,8)$ be a pair of $4$-cycles in $\dcd^*(\sigma)$. We first use Bertram's construction and write $\sigma\sigma_1^{-1}\sigma_2^{-1}$ as $\sigma\sigma_1^{-1}\sigma_2^{-1} = \rho_1\rho_2 \in \C_{\ell'}^2$, where $\ell' = \frac{m_\sigma + n_\sigma -10}{2}$ or $\ell' = \frac{m_\sigma + n_\sigma -10}{2} + 1$, depending on whether $m_\sigma \geq \ell + 2k$ or $m_\sigma = \ell + 2k-1$. Let $\rho_1 = (a_1, a_2, \dots, a_{\ell'-1}, a)$ and $\rho_2 = (a, b_2, b_3, \dots, b_{\ell'})$. Then we write $\sigma = \rho \tau$, where for $k=1$, $\rho =(a_1, a_2, \dots, a_{\ell'-1}, a, 6,1) (2,5) (3,8) (4,7)$ and $\tau = (1,5,a, b_2, b_3, \dots, b_{\ell'}) (2,8)(3,7) (4,6)$; for $k=2$, $\rho=(a_1, a_2, \dots, a_{\ell'-1}, a, 2, 5, 6) (3,8) (4,7)$ and $\tau = (6,4,1,a, b_2, b_3, \dots, b_{\ell'}) (2,8)(3,7)$; for $k=3$, $\rho = (a_1, a_2, \dots, a_{\ell'-1}, a, 2, 5, 6,7) (3,8)$ and $\tau = (7,3,4,1, a, b_2, b_3, \dots, b_{\ell'}) (2,8)$. Thus, $\sigma \in \C_{\lambda}^2$ for $0<k \leq 3$ and $n_{4} \geq 2$. Hence, we can assume $n_{4} \leq 1$.
        

        \medskip

        Now, for some $i$, let $\sigma_i$ be such that $\ell_i$ is odd and $\ell_i \geq 5$. We let $\sigma_i:= (1,2,\dots,\ell_i)$. Since all cycles in $\dcd^*(\sigma)$ have length $\leq 2k+2$, $\ell_i < 2k+3$ and hence $\frac{\ell_i -3}{2} < k$. We write $\sigma_i = \rho \tau$, where $\rho =\left(\frac{\ell_i+3}{2}, \frac{\ell_i+1}{2}, 1\right) (2,\ell_i) (3,\ell_i-1) \cdots \left(\frac{\ell_i-1}{2}, \frac{\ell_i+5}{2}\right)$ and $\tau=\left(1, \ell_i, \frac{\ell_i+1}{2}\right) (2, \ell_i-1)(3,\ell_i-2) \cdots \left(\frac{\ell_i-1}{2}, \frac{\ell_i+3}{2}\right)$. Thus, $\sigma_i \in \C_{\lambda_i}^2$, where $\lambda_i= (3, 2^{\frac{\ell_i -3}{2}}, 1^{n-\ell_i})$. Since $\ell_i \geq 5$, $\frac{m_{\sigma_i} + n_{\sigma_i}}{2} \geq \frac{5+1}{2} =3$. Thus, $\ell \geq \frac{m_\sigma +n_\sigma}{2} \implies \ell-3 \geq \frac{(m_\sigma-m_{\sigma_i}) +(n_\sigma-n_{\sigma_i})}{2}$. Since $\frac{\ell_i -3}{2} < k$, it is enough to show that  $\sigma\sigma_i^{-1} \in \C_{\lambda'}^2$, where $\lambda' = (\ell -3, 2^{k-\frac{\ell_i -3}{2}}, 1^{n-\ell +3 -2\left(k-\frac{\ell_i -3}{2}\right)})$. Indeed, if that is the case, then by a suitable concatenation of symbols, we get $\sigma \in \C_{\lambda}^2$. In order to show that $\sigma\sigma_i^{-1} \in \C_{\lambda'}^2$, first assume that there exists an cycle of odd length $\geq 2\left(k-\frac{\ell_i -3}{2}\right)+3$. In this case, we apply Lemma \ref{lemma: large odd cycle} to conclude that  $\sigma\sigma_i^{-1} \in \C_{\lambda'}^2$ and consequently $\sigma \in \C_{\lambda}^2$. If there does not exist a cycle of odd length $\geq 2\left(k-\frac{\ell_i -3}{2}\right)+3$, then we again pick a cycle of odd length $\geq 5$ (if it exists), and repeat the above process. An iterative repetition of this process allows us to assume that $n_i =0$, if $i$ is odd and $i\geq 5$.

        \medskip

        By assumption, all cycles in $\dcd^*(\sigma)$ have length $\leq 2k+2$. We now proceed to reduce to the case when there is at most one cycle of even length $\geq 6$ in $\dcd^*(\sigma)$. If that doesn't already hold for $\sigma$, we proceed as follows:

        \medskip

        \noindent \textbf{Step 1:} For some $j$, let $\sigma_j:= (1,2,\dots,\ell_j)$ be such that $\ell_j$ is even and $\ell_j\geq 6$. Since $\ell_j < 2k+4$, $\frac{\ell_j -4}{2} < k$. Thus, for any $k' \leq \frac{\ell_j -4}{2}$, we can write $\sigma_j$ as  $\sigma_j = \rho_1\tau_1 = \tau_2\rho_2$, where $\rho_1 = \left(\frac{\ell_j+2}{2}, \frac{\ell_j+4}{2}, \cdots, \ell_j-k', k'+2,1\right)(2,\ell_j)(3,\ell_j-1)\cdots  (k'+1, \ell_j-k'+1)$, 

        \noindent $\tau_1 = \left(1,\ell_j,k'+2,k'+3,\cdots,\frac{\ell_j}{2}\right)(2, \ell_j-1)(3,\ell_j-2) \cdots  (k'+1, \ell_j-k')$, 

        \noindent $\tau_2 = \left(\frac{\ell_j+4}{2}, \frac{\ell_j+6}{2}, \cdots, \ell_j-k', k'+2,1\right)(2,\ell_j)(3,\ell_j-1)\cdots  (k'+1, \ell_j-k'+1)$ and 

        \noindent $\rho_2 = \left(1,\ell_j,k'+2,k'+3,\cdots,\frac{\ell_j}{2}, \frac{\ell_j+2}{2}\right)(2, \ell_j-1)(3,\ell_j-2) \cdots  (k'+1, \ell_j-k')$.

        \noindent Note that $\rho_1,\rho_2\in \C_{\lambda_j}$ and $\tau_1, \tau_2 \in  \C_{\lambda_j'}$, for $\lambda_j= (\frac{\ell_j-2k'+2}{2}, 2^{k'}, 1^{n-\frac{\ell_j-2k'+2}{2} - 2k'})$ and $\lambda_j'= (\frac{\ell_j-2k'+4}{2}, 2^{k'}, 1^{n-\frac{\ell_j-2k'+4}{2} - 2k'})$. Since $\sigma$ contains no cycle of even length $\geq 2k+4$, $\frac{\ell_j-4}{2}<k$. So, we take $k' = \frac{\ell_j-4}{2}<k$, and hence we write $\sigma_j = \rho_1\tau_1 = \tau_2\rho_2$, where $\rho_1,\rho_2\in \C_{\lambda_j}$ and $\tau_1, \tau_2 \in  \C_{\lambda_j'}$, for $\lambda_j= (3, 2^{\frac{\ell_j-4}{2}}, 1^{n-3- 2\left(\frac{\ell_j-4}{2}\right)})$ and $\lambda_j'= (4, 2^{\frac{\ell_j-4}{2}}, 1^{n-4 - 2\left(\frac{\ell_j-4}{2}\right)})$.

        \medskip

        \noindent \textbf{Step 2:} If $\sigma\sigma_j^{-1}$ contains a cycle $\sigma_{m}$ of even length $\ell_{m}\geq 2\left(k-\frac{\ell_j -4}{2}\right)+4$, then for $\sigma_m$, we take $k' = k-\frac{\ell_j -4}{2}$, and proceed similar to Step $1$ and write $\sigma_{m}$ as $\sigma_{m} = \rho_1'\tau_1' = \tau_2'\rho_2'$, where $\rho_1',\rho_2'\in \C_{\lambda_{m}}$ and $\tau_1', \tau_2' \in  \C_{\lambda_{m}'}$, for $\lambda_{m}= (\frac{\ell_m-2\left(k-\frac{\ell_j -4}{2}\right)+2}{2}, 2^{k-\frac{\ell_j -4}{2}}, 1^{n-\frac{\ell_m-2\left(k-\frac{\ell_j -4}{2}\right)+2}{2}-2\left(k-\frac{\ell_j -4}{2}\right)})$, and $\lambda_{m}'= (\frac{\ell_m-2\left(k-\frac{\ell_j -4}{2}\right)+4}{2}, 2^{k-\frac{\ell_j -4}{2}}, 1^{n-\frac{\ell_m-2\left(k-\frac{\ell_j -4}{2}\right)+4}{2}-2\left(k-\frac{\ell_j -4}{2}\right)})$. By a suitable concatenation of symbols, we get $\sigma_{j} \sigma_{m} \in \C_{\lambda'}^2$, where $\lambda'= (\ell', 2^{k}, 1^{n-\ell'-2k})$, $\ell'= \frac{4+\ell_m-2\left(k-\frac{\ell_j -4}{2}\right)+2}{2} = \frac{\ell_j + \ell_m -2k +2}{2}$. Now, for $\sigma\sigma_1^{-1} \sigma_2^{-1}$, $\frac{m_{\sigma\sigma_1^{-1} \sigma_2^{-1}} + n_{\sigma\sigma_1^{-1} \sigma_2^{-1}}}{2} = \frac{m_\sigma +n_\sigma -\ell_j - \ell_m -2}{2}$. Then by Bertram's construction, we write $\sigma\sigma_1^{-1} \sigma_2^{-1}$ as a product of two cycles of length $\ell''$, where $\ell''= \frac{m_\sigma +n_\sigma -\ell_j - \ell_m -2}{2}$ or $\ell''= \frac{m_\sigma +n_\sigma -\ell_j - \ell_m}{2}$, depending on whether $m_\sigma \geq \ell + 2k$ or $m_\sigma = \ell + 2k -1$. Thus, by a proper concatenation of symbols, $\sigma \in \C_\mu^2$, where $\mu = (\ell'+\ell'', 2^k, 1^{n-\ell'-\ell'' -2k})$. Since $\ell'+\ell'' \leq \frac{\ell_j + \ell_m -2k +2}{2} + \frac{m_\sigma +n_\sigma -\ell_j - \ell_m}{2} = \frac{m_\sigma +n_\sigma -2k+2}{2} \leq \frac{m_\sigma +n_\sigma}{2}$, $\sigma \in \C_{\lambda}^2$. 

        \medskip
        
        \noindent \textbf{Step 3:} If $\sigma\sigma_j^{-1}$ contains no cycle of even length $\geq 2\left(k-\frac{\ell_j -4}{2}\right)+4$, then we again choose some cycle of even length $\sigma_m$ (if it exists) with $\ell_m \geq 6$, take $k' = \frac{\ell_m -4}{2}$, and write $\sigma_m$ as $\sigma_{m} = \rho_1'\tau_1' = \tau_2'\rho_2'$, where $\rho_1',\rho_2'\in \C_{\lambda_{m}}$ and $\tau_1', \tau_2' \in  \C_{\lambda_{m}'}$, for $\lambda_{m}= (3, 2^{\frac{\ell_m -4}{2}}, 1^{n-3-2\left(\frac{\ell_m -4}{2}\right)})$ and $\lambda_{m}'= (4, 2^{\frac{\ell_m -4}{2}}, 1^{n-4-2\left(\frac{\ell_m -4}{2}\right)})$. Consequently, $\sigma_{j} \sigma_{m} \in \C_{\lambda'}^2$, where $\lambda'= (7, 2^{\frac{\ell_j + \ell_m -8}{2}}, 1^{n-7-2\left(\frac{\ell_j + \ell_m -8}{2}\right)})$. As $\ell -7 \geq \frac{m_\sigma + n_\sigma -14} {2}\geq \frac{m_\sigma -\ell_j-\ell_m + n_\sigma -2} {2} =\frac{m_{\sigma\sigma_j^{-1} \sigma_m^{-1}} + n_{\sigma\sigma_j^{-1} \sigma_m^{-1}}}{2}$ and $\sigma\sigma_j^{-1}\sigma_m^{-1}$ is an even permutation, it is enough to show that $\sigma\sigma_j^{-1} \sigma_m^{-1} \in \C_{\lambda'}^2$, where $\lambda'= (\ell-7, 2^{k-\frac{\ell_j + \ell_m -8}{2}}, 1^{n-\ell+7 - 2\left(k-\frac{\ell_j + \ell_m -8}{2}\right)})$. In order to show that we first assume $\sigma\sigma_j^{-1}\sigma_m^{-1}$ contains a cycle of even length $\geq 2\left( k-\frac{\ell_j + \ell_m -8}{2}\right) + 4 $,  whence by Lemma \ref{lemma: large even cycle}, we obtain $\sigma\sigma_j^{-1} \sigma_m^{-1} \in \C_{\lambda'}^2$ and consequently $\sigma \in \C_{\lambda}^2$. If $\sigma\sigma_j^{-1}\sigma_m^{-1}$ contains no cycle of even length $\geq 2\left( k-\frac{\ell_j + \ell_m -8}{2}\right) + 4$, then we replace $\sigma$ by $\sigma\sigma_j^{-1}\sigma_m^{-1}$ and repeat the same process for $\sigma\sigma_j^{-1}\sigma_m^{-1}$ from Step $1$.

        \medskip
        
        An iterative repetition of the above process allows us to assume that there is at most one cycle of even length $\ell_i \geq 6$ in $\dcd^*(\sigma)$. In addition to this, we already have $n_2,n_3,n_4 \leq 1$, and $n_{\ell_i} =0$ if $\ell_i$ is odd and $\geq 5$. Since $\sigma \in A_n$ and $\ell +2k -1 \leq m_\sigma$, we are left with the following possibilities of $\sigma$. 

        \medskip
        
        \noindent \textbf{(a)} $\sigma = \sigma_1\sigma_2$, where $\sigma_1$ is a $2$-cycle and $\sigma_2$ is a $4$-cycle. By the assumption $\ell \geq \frac{m_\sigma + n_\sigma}{2}$, we have $\ell \geq 4$. We also have $\ell + 2k -1\leq m_\sigma = 6 \implies \ell +2k \leq 7$. Thus, the possible choices for $\lambda$ are $(4,2,1^{n-6})$ and $(5,2,1^{n-7})$. For these choices, let $\sigma = (1,2)(3,4,5,6)$. Then we write $\sigma$ as $\sigma = \rho \tau \in \C_\lambda^2$ by taking $\rho = (4,6,5,1)(2,3)$, $\tau = (2,5,4,6)(1,3)$ if $\lambda=(4,2,1^{n-6})$; and $\rho =(4,7,6,5,1)(2,3)$, $\tau = (2,5,7,4,6)(1,3)$ if $\lambda = (5,2,1^{n-7})$.

        \medskip
        
        \noindent \textbf{(b)} $\sigma = \sigma_1\sigma_2\sigma_3$, where $\sigma_1$ is a $2$-cycle, $\sigma_2$ is a $3$-cycle and $\sigma_3$ is a $4$-cycle. By the assumption $\ell \geq \frac{m_\sigma + n_\sigma}{2}$, we have $\ell \geq 6$. We also have $\ell + 2k -1\leq m_\sigma = 9 \implies \ell +2k \leq 10$. Thus, the possible choices for $\lambda$ are $(6,2, 1^{n-8}),~(6,2^2, 1^{n-10}),~(7,2, 1^{n-9})$ and $(8,2, 1^{n-10})$. Let $\sigma = (1,2)(3,4,5)(6,7,8,9)$. Then we write $\sigma$ as $\sigma = \rho \tau \in \C_\lambda^2$ by taking $\rho= (7,8,5,9,3,1)(2,6)$, $\tau = (2,3,4,8,5,9)(1,6)$ if $\lambda = (6,2, 1^{n-8})$; $\rho= (7,3,9,5,10,1)(2,6)(4,8)$, $\tau = (2,10,5,7,4,9)(1,6)(3,8)$ if $\lambda = (6,2^2, 1^{n-10})$; $\rho= (7,4,8,5,9,3,1)(2,6)$, $\tau = (2,3,7,4,8,5,9)(1,6)$ if $\lambda = (7,2, 1^{n-9})$; and $\rho= (7,4,8,5,9,3,10,1)(2,6)$, $\tau = (2,10,3,7,4,8,5,9)(1,6)$ if $\lambda = (8,2, 1^{n-10})$.

        \medskip
        
        \noindent \textbf{(c)} $\sigma = \sigma_1\sigma_2$ such that $\sigma_1$ is a cycle of length $\ell_1$ where $\ell_1=2$ or $4$, and $\sigma_2$ is a cycle of length $\ell_2$ where $\ell_2$ is even and $\geq 6$. By the assumption $\ell \geq \frac{m_\sigma + n_\sigma}{2}$, we have $\ell \geq \frac{\ell_1 +\ell_2+2}{2}$. In addition, $\ell + 2k -1\leq m_\sigma = \ell_1 +\ell_2 \implies \ell +2k \leq \ell_1 +\ell_2 +1$. We also have the assumption that all cycles in $\sigma$ have length $\leq 2k+2$. Thus, $\ell + 2k \leq \ell_1 + 2k+2 +1 \implies \ell \leq \ell_1 +3$. Therefore, $\frac{\ell_1 +\ell_2 +2}{2} \leq \ell \leq \ell_1 +3$. For $\ell_1 =2$, this leads to $\frac{4 +\ell_2}{2} \leq \ell \leq 5$. Thus, for $\ell_1 =2$ and $\ell_2 \geq 6$, the only possible choice for the pair $(\ell_2, \ell)$ is $(6,5)$. Since $\ell_2 < 2k+4$, $k > \frac{\ell_2-4}{2}$. Further, by using $\ell +2k-1 \leq m_\sigma$, we get $k \leq \frac{m_\sigma+1 - \ell}{2}$. Thus, $\frac{\ell_2-4}{2} < k \leq \frac{m_\sigma+1 - \ell}{2}$. Consequently, for $\ell_1 =2$, $\ell_2 =6$ and $\ell=5$, we get $1 < k \leq 2\implies k=2$. Thus, the only possible choice for $\lambda$ is $(5,2^2,1^{n-9})$. For this choice, we let $\sigma:= (1,2)(3,4,5,6,7,8)$. Choose $\rho = (4,8,6,9,1)(2,3)(5,7)$ and $\tau = (2,9,6,5,8)(1,3)(4,7)$. Then $\sigma = \rho \tau \in \C_\lambda^2$, where $\lambda=(5,2^2,1^{n-9})$.

        \medskip

        Now, for $\ell_1 =4$, the relation $\frac{\ell_1 +\ell_2+2}{2} \leq \ell \leq \ell_1 +3$ becomes $\frac{6 +\ell_2}{2} \leq \ell \leq 7$. Hence, for $\ell_2 \geq 6$ and $\ell_2$ being even, the possible choices for the pair $(\ell_2, \ell)$ are $(6,6)$, $(6,7)$ and $(8,7)$. By the relation $\frac{\ell_2-4}{2} < k \leq \frac{m_\sigma+1 - \ell}{2}$, we get $k=2$ if $(\ell_2, \ell) = (6,6), (6,7)$; and $k=3$ if $(\ell_2,\ell) = (8,7)$. Consequently, for $\ell_1 =4$, the possible choices for the triplet $(\ell_2, \ell, k)$ are $(6,6,2)$, $(6,7,2)$ and $(8,7,3)$. The corresponding choices for $\lambda$ are $(6,2^2,1^{n-10})$, $(7,2^2,1^{n-11})$, and $(7,2^3,1^{n-13})$. For $(\ell_2, \ell, k)=(6,6,2)$ and $(6,7,2)$, we let $\sigma = (1,2,3,4) (5,6,\dots,10)$, and write it as a product $\rho \tau \in \C_\lambda^2$ by choosing $\rho = (5,7,4,8,6,1)(2,10)(3,9)$, $\tau = (3,7,4,6,5,8)(1,10)(2,9)$ if $\lambda = (6,2^2,1^{n-10})$; and $\rho = (5,7,11,4,8,6,1)(2,10)(3,9)$, $\tau = (3,11,7,4,6,5,8)(1,10)(2,9)$ if $\lambda = (7,2^2,1^{n-11})$. For $(\ell_2, \ell, k)= (8,7,3)$, we let $\sigma = (1,2,3,4) (5,6,\dots,12)$, and write it as a product $\rho \tau \in \C_\lambda^2$ by choosing $\rho = (5,8,7,9,6,13,1)(2,12)(3,11)(4,10)$, $\tau = (4,13,6,8,7,5,9)(1,12)(2,11)(3,10)$, where $\lambda = (7,2^3,1^{n-13})$.

        \medskip

        \noindent \textbf{(d)} $\sigma = \sigma_1\sigma_2\sigma_3$ such that $\sigma_1$ is a $3$-cycle, $\sigma_2$ is a cycle of length $\ell_2$ where $\ell_2 = 2$ or $4$, and  $\sigma_3$ is a cycle of length $\ell_3$ where $\ell_3$ is even and $\geq 6$. By $\ell \geq \frac{m_\sigma + n_\sigma}{2}$, we have $\ell \geq \frac{\ell_2 +\ell_3 +6}{2}$. In addition, $\ell + 2k -1\leq m_\sigma = \ell_2 +\ell_3 +3 \implies \ell +2k \leq \ell_2 +\ell_3 +4$. We also have the assumption that all cycles in $\sigma$ have length $\leq 2k+2$. Thus, $\ell + 2k \leq \ell_2 + 2k+2 +4 \implies \ell \leq \ell_2 +6$. Evidently, $\frac{\ell_2 +\ell_3 +6}{2} \leq \ell \leq \ell_2 +6$. For $\ell_2 =2$, this leads to $\frac{8 +\ell_3}{2} \leq \ell \leq 8$. For $\ell_3 \geq 6$ and $\ell_3$ being even, the possible choices for the pair $(\ell_3, \ell)$ when $\ell_2 =2$ are $(6,7)$, $(6,8)$ and $(8,8)$. Since $\ell_3 < 2k+4$, $k > \frac{\ell_3-4}{2}$. Further, by using $\ell +2k-1 \leq m_\sigma$, we get $k \leq \frac{m_\sigma+1 - \ell}{2}$. Thus, $\frac{\ell_3-4}{2} < k \leq \frac{m_\sigma+1 - \ell}{2}$. Consequently, for $\ell_2 =2$, the possible choices for the triplet $(\ell_3, \ell, k)$ are $(6,7,2)$, $(6,8,2)$ and $(8,8,3)$. The corresponding choices for $\lambda$ are $(7,2^2,1^{n-11})$, $(8,2^2,1^{n-12})$, and $(8,2^3,1^{n-14})$. For $(\ell_2, \ell, k)=(6,7,2)$ and $(6,8,2)$, we let $\sigma = (1,2,3) (4,5) (6,7,\dots,11)$, and write it as a product $\rho \tau \in C_\lambda^2$ by choosing $\rho = (6,8,5,9,7,4,1)(2,11)(3,10)$, $\tau = (3,4,8,5,7,6,9)(1,11)(2,10)$ if $\lambda = (7,2^2,1^{n-11})$; and $\rho = (6,8,5,9,7,4,12,1)(2,11)(3,10)$, $\tau = (3,12,4,8,5,7,6,9)(1,11)(2,10)$ if $\lambda = (8,2^2,1^{n-12})$. For $(\ell_2, \ell, k)= (8,8,3)$, we let $\sigma = (1,2,3) (4,5) (6,7,\dots,13)$. Here $\lambda = (8,2^3,1^{n-14})$, so we write $\sigma=\rho \tau \in \C_\lambda^2$ by choosing $\rho = (6,8,4,11,7,9,14,1)(2,13)(3,12)(5,10)$ and $\tau = (3,14,9,5,8,7,6,11)(1,13)(2,12)(4,10)$.

      \medskip Now, for $\ell_2 =4$, the relation $\frac{\ell_2 +\ell_3 +6}{2} \leq \ell \leq \ell_2 +6$ becomes $\frac{10 +\ell_3}{2} \leq \ell \leq 10$. Hence, for $\ell_3 \geq 6$ and $\ell_3$ being even, the possible choices for the pair $(\ell_3, \ell)$ are $(6,8)$, $(6,9)$, $(6,10)$, $(8,9)$, $(8,10)$ and $(10,10)$. By $\frac{\ell_3-4}{2} < k \leq \frac{m_\sigma+1 - \ell}{2}$, we get $k=2,3$ if $(\ell_3,\ell) =(6,8)$; $k=2$ if $(\ell_3,\ell) =(6,9),(6,10)$; $k=3$ if $(\ell_3,\ell) =(8,9), (8,10)$; and $k=4$ if $(\ell_3,\ell) = (10,10)$. Thus, for $\ell_2 =4$, the possible choices for $(\ell_3, \ell, k)$ are $(6,8,2)$, $(6,8,3)$, $(6,9,2)$, $(6,10,2)$, $(8,9,3)$, $(8,10,3)$ and $(10,10,4)$. It is enough to show that $\sigma \in \C_\lambda^2$ for the choices $(6,8,2)$, $(6,8,3)$, $(8,9,3)$ and $(10,10,4)$ because the conclusion follows for other choices by increasing the length $\ell$ accordingly. The corresponding choices for $\lambda$ are $(8,2^2,1^{n-12})$, $(8,2^3,1^{n-14})$, $(9,2^3,1^{n-15})$, and $(10,2^4,1^{n-18})$. For $(\ell_3, \ell, k)=(6,8,2)$ and $(6,8,3)$, we let $\sigma = (1,2,3) (4,5,6,7) (8,9,\dots,13)$, and choose $\rho = (8,4,5,7,11,6,10,1)(2,13)(3,12)$, $\tau = (3,10,7,8,9,6,5,11)(1,13)(2,12)$ if $\lambda = (8,2^2,1^{n-12})$; and $\rho = (8,6,4,11,7,9,14,1)(2,13)(3,12)(5,10)$, $\tau = (3,14,9,5,8,7,6,11)(1,13)(2,12)(4,10)$ if $\lambda = (8,2^3,1^{n-14})$ to show that $\rho \tau \in \C_\lambda^2$. For $(\ell_3, \ell, k)= (8,9,3)$, we have $\lambda=(9,2^3,1^{n-15})$. Thus, we let $\sigma= (1,2,3) (4,5,6,7) (8,9,\dots,15)$, and write $\sigma=\rho \tau \in \C_\lambda^2$ by choosing $\rho = (8,4,13,9,7,10,6,11,1)(2,15)(3,14)(5,12)$ and $\tau = (3,11,5,10,6,9,7,8,13)(1,15)(2,14)(4,12)$. Further, for $(\ell_3, \ell, k)= (10,10,4)$, we have $\lambda=(10,2^4,1^{n-18})$. In this case, we let $\sigma=(1,2,3) (4,5,6,7) (8,9,\dots,17)$, and choose $\rho = (8,10,4,15,9,11,7,12,18,1)(2,17)(3,16)(5,14)(6,13)$ and $\tau = (3,18,12,6,11,7,10,9,8,15)(1,17)(2,16)(4,14)(5,13)$ to show that $\sigma=\rho \tau \in \C_\lambda^2$.  
    \end{proof}

    \begin{proof}[\textbf{Proof of Theorem \ref{prod_classes_alternating}}]
        We divide the proof into three cases based on $m_\sigma$.

        \medskip

        \noindent \textbf{Case 1:} $m_\sigma \leq \ell$. \quad In this case, we first write $\sigma$ as a product of two $\ell$-cycles, say $\tau_1$ and $\tau_2$ in $A_\ell$, using Bertram's construction. Let $a_1,a_2,\dots,a_{2k} \in \{1,2,\dots, n\}$ be such that $a_i \not\in \supp(\tau_1) \cup \supp(\tau_2)$, for all $i$. Then $$\sigma = \tau_1.\tau_2= \tau_1 (a_1,a_2) (a_3,a_4) \dots (a_{2k-1}, a_{2k}).\tau_2(a_1,a_2)(a_3,a_4)\dots (a_{2k-1}, a_{2k}) \in \C_\lambda^2.$$
        \noindent \textbf{Case 2:} $\ell +2k-1 \leq m_\sigma \leq n$. \quad This case follows directly from Lemma \ref{lemma: large odd cycle}, Lemma \ref{lemma: large even cycle}, and Lemma \ref{lemma: all cycles are small}.

        \noindent \textbf{Case 3:} $\ell < m_\sigma \leq \ell +2k-2$. \quad
        It is enough to assume that $\ell +2k = n$. Then $m_\sigma$ has $\ell +2k -m_\sigma$ fixed points. Let $r:= \ell +2k -m_\sigma = 2m +\epsilon$ for some $m$ and $\epsilon = 0,1$. By Case $2$, we write $\sigma = \rho_1\rho_2 \in \C_{\lambda'}^2$ in $A_{n-2m}$, where $\lambda'= (\ell, 2^{k-m})$. Let $\alpha_1,\alpha_2,\dots, \alpha_{2m} \notin \supp(\rho_1) \cup \supp(\rho_2)$. Then $\sigma = \rho_1 (\alpha_1,\alpha_2)\cdots (\alpha_{2m-1}, \alpha_{2m})\rho_2 (\alpha_1,\alpha_2)\cdots (\alpha_{2m-1}, \alpha_{2m}) \in \C_{\lambda}^2$. 

        \noindent In particular if $l\geq \lfloor \frac{3n}{4} \rfloor$, then for any $\sigma\in A_n$, $\frac{m_{\sigma}+n_{\sigma}}{2}\leq l$, whence $\sigma \in \C_{\lambda}^2$. Thus, $\C_{\lambda}^2=A_n$ and our proof is complete.
    \end{proof}

    \section{2-generation of conjugacy classes of derangements in $S_n$ and $A_n$}\label{proof_third_theorem}

    As mentioned in the introduction, before proving Theorem \ref{generation_by_conjugate_derangements}, we give a brief account of some results on the generation of $S_n$ and $A_n$ by conjugate elements. Let $\C_{\lambda}$ be a conjugacy class of $S_n$ indexed by a partition $\lambda\vdash n$. Define $\displaystyle \delta(\C_{\lambda})=\Big \lceil \frac{n-1}{n-l(\lambda)} \Big\rceil$. Note that if $\sigma \in \C_{\lambda}$, then $\displaystyle \delta(\C_{\lambda})=\Big\lceil \frac{n-1}{m_{\sigma}-n_{\sigma}} \Big\rceil$. It can be shown that at least $\delta(\C_{\lambda})$ many elements of $\C_{\lambda}$ are required to generate a transitive subgroup of $S_n$ and conversely it suffices to take  $\delta(\C_{\lambda})$ many elements of $\C_{\lambda}$ to generate a transitive subgroup of $S_n$ (see \cite{ro2}). This shows that at least $\delta(\C_{\lambda})$ many elements are required to generate $S_n$ or $A_n$. A general result in this direction is due to Hall-Liebeck-Seitz. In \cite[Theorem 3]{hls}, they prove that $\delta(\C_{\lambda})+2$ elements of $\C_{\lambda}$ generate $S_n$ or $A_n$, unless $\C_{\lambda}$ is the conjugacy class of fixed-point-free involutions. A few results are also available for specific conjugacy classes. In \cite{sm}, the authors show that $\delta(\C_l)$ many elements of the conjugacy class $\C_l$ of $l$-cycles are enough to generate $S_n$ or $A_n$ when $2\leq l<n$. They also show that a pair of $n$-cycles generates $S_n$ or $A_n$ (note that $\delta(\C_n)=1$). An asymptotic result in this direction is due to Beasley-Brenner-Erd\H{o}s-Szalay-Williamson. In \cite{bbesw}, they prove that almost all conjugacy classes of the alternating group $A_n$ contain a pair of generators. Rodgers defined the notion of conjugacy-generating-number of a non-trivial conjugacy class $C$ of a finite non-abelian simple group $G$ as the minimum number of elements of $C$ required to generate $G$. Let it be denoted by $\gn(C)$. In \cite{ro1}, the author observed that a simple argument using the $\frac{3}{2}$-generation of finite simple groups yields $\gn(C)\leq \cn(C)+1$, where $\cn(C)$ is the covering number of $C$ in $G$. Moreover, when $G=A_n$ and $n\geq 5$, he showed that $\cn(C)\leq \frac{3}{2}\gn(C)+4$ (see \cite[Theorem 6]{ro1}). In this language, Theorem \ref{generation_by_conjugate_derangements} states that if $C$ is a conjugacy class of derangements in $A_n$, then $\gn(C)=2$, unless $C$ consists of fixed-point-free involutions. Now we proceed towards proving Theorem \ref{generation_by_conjugate_derangements}.

    \medskip

    We recall the well-known notion of a primitive action. Let $G$ be a transitive permutation group on a set $X$. A non-empty subset $B\subseteq X$ is called a block of $G$ if $g.B=B$ or $g.B\cap B =\varnothing$ for all $g\in G$. A block is non-trivial if it is neither a singleton nor equals $X$. It is easy to see that if $B$ is a block then $g.B$ is a block as well. $G$ is called imprimitive if there exists a non-trivial block $B$, in which case $B$ is called a block of imprimitivity. In this case, we get a collection of non-trivial blocks, say $\{B_1,\ldots, B_k\}$ such that $\displaystyle B_i\cap B_j=\varnothing$, $\displaystyle\sqcup_{i}B_i=X$, and each $g\in G$ permutes the $B_i$'s. Then, $\mathcal{B}=\{B_1,B_2,\ldots, B_k\}$ is called a complete block system for $G$. We also conclude that all blocks of $G$ have the same size. $G$ is called primitive if there does not exist a non-trivial block for $G$. An old result of Jordan (see \cite{wi}) states that if $G$ is a primitive subgroup of $S_n$ and $G$ contains a cycle of prime length $p$ such that $2\leq p\leq n-3$, then $G$ contains $A_n$. To establish the 2-generation of conjugacy classes of derangements, we will invoke the following result of Jones, which removes the primality condition from Jordan's theorem. It is worth noting that Jones uses the CFSG to establish the result. 

    \begin{theorem}\cite[Corollary 1.3]{Jo}\label{Jones_Jordan}
        Let $G$ be a primitive subgroup of $S_n$ and $G$ contains a cycle of length $l$ such that $2\leq l  \leq n-3$. Then, $G$ contains $A_n$.
    \end{theorem}

    \noindent For convenience, let $\mathscr{P}_t(n)$ denote the set of all partitions of $n$ with all parts greater than $t$. Then, the conjugacy classes of derangements in $S_n$ are completely parameterized by partitions $\lambda\in \mathscr{P}_1(n)$. We are now ready to prove Theorem \ref{generation_by_conjugate_derangements}. We first prove the following more general statement.
    \begin{theorem}\label{two_generation_derangements}
        Let $\lambda=(\lambda_1,\ldots,\lambda_l)\in \mathscr{P}_1(n)$ and $\lambda\neq (2^{n/2})$, where $n\geq 5$. Let $\C_{\lambda}$ be the conjugacy class of \(S_n\) indexed by $\lambda$. Let $\sigma=
        \displaystyle \prod_{i=1}^{l}(a_{i,1},a_{i,2},\ldots,a_{i,\lambda_i}) \in \C_{\lambda}$. Let $2\leq r<\lambda_i$
        and define $\tau=(a_{i,j},\sigma(a_{i,j}), \cdots, \sigma^{r-1}(a_{i,j}))$.
        Further, define
        \[
            \rho=(a_{1,1},\ldots,a_{1,\lambda_1},
            a_{2,1},\ldots,a_{2,\lambda_2},
            \ldots, a_{i-1,\lambda_{i-1}}, a_{i,j},\sigma(a_{i,j}), \cdots, \sigma^{\lambda_i-1}(a_{i,j}), a_{i+1,1}, \ldots, a_{l,\lambda_l}).
        \]
        Then, $H:=\langle \tau \sigma \tau^{-1},\rho \sigma \rho^{-1}\rangle$ contains $A_n$.
    \end{theorem}

    \noindent With the notations of the Theorem \ref{two_generation_derangements}, 

    \noindent (1) for $1\leq k\leq l$, set
            \[
                    h_k=\begin{cases}
                    a_{k,1} & \text{if}\;\;k\neq i\\
                    a_{k,j} &  \text{if}\;\;k=i
                    \end{cases},\qquad \text{and} \qquad  t_k=\begin{cases}
                    a_{k,\lambda_k} & \text{if}\;\;k\neq i\\
                    \sigma^{\lambda_i-1}(a_{k,j}) & \text{if}\;\;k=i
            \end{cases}.\]
        We call $h_k$ and $t_k$ the head and tail of the $k$-th cycle (of length $\lambda_k$) of $\sigma$, respectively.
        
    \noindent (2) Set $\displaystyle x:=\tau \sigma \tau^{-1}=\prod_{\substack{k=1\\k\neq i}}^{l}(a_{k,1},\ldots,a_{k,\lambda_k})(\sigma(a_{i,j}),\cdots,\sigma^{r-1}(a_{i,j}), a_{i,j}, \sigma^{r}(a_{i,j}), \cdots,\sigma^{\lambda_{i}-1}(a_{i,j}))$.
        \begin{align*} \text{Set } y:&=\rho \sigma \rho^{-1}\\
              &=\prod_{\substack{k=1\\k\neq i,i-1}}^{l}(a_{k,2},\ldots,a_{k,\lambda_k},a_{k+1,1})(a_{i-1,2},\cdots,a_{i-1,\lambda_{i-1}},a_{i,j})(\sigma(a_{i,j}),\cdots,\sigma^{\lambda_{i}-1}(a_{i,j}),a_{i+1,1}),\end{align*} where $a_{l+1,1}=a_{1,1}$.

    \begin{lemma}\label{lemma_cycle_length}
        Let $\sigma, \tau, \rho$ be as in the statement of Theorem \ref{two_generation_derangements}. Let $x=\tau\sigma\tau^{-1}$ and $y=\rho \sigma \rho^{-1}$. Then, 
        \[g:=x^{-1}y=(\sigma^{r-1}(a_{i,j}),h_i,h_{i-1},\ldots,h_1,h_l,h_{l-1},\ldots,h_{i+1},t_i,t_{i+1},\ldots,t_l, t_1,t_2,\ldots,t_{i-1}),\]
        and hence, \(g\) is a cycle of length \(2l+1\).
    \end{lemma}

    \begin{proof}
        A direct computation.
    \end{proof}

    \begin{lemma}\label{primitive_1}
        Let $\lambda, \sigma, \tau, \rho$ be as in Theorem \ref{two_generation_derangements}. Assume further that $\lambda\neq(3,2^t)\vdash n$. Then $H:=\langle \tau \sigma \tau^{-1},\rho \sigma \rho^{-1}\rangle$ is primitive.
    \end{lemma}

    \begin{proof}
        We prove this in several steps. Set $x=\tau \sigma \tau^{-1},\;y=\rho \sigma \rho^{-1},\; g=x^{-1}y$. Assume that $H:=\langle x,y \rangle$ is imprimitive and let \(\mathcal B\) be a non-trivial block system for \(H\). Set $\Omega:=\{\sigma^{r-1}(a_{i,j}), h_1,\ldots, h_{l}, t_1, \ldots, t_{l}\}$. Note that $\Omega$ is the set of points moved by $g$ and $\Omega^c$ is the set of all points that are fixed by $g$. Since $\lambda \neq (3,2^t)$, it follows that $\Omega^c\neq \varnothing$.

        \medskip

        \noindent \textbf{Claim 1}: Every non-trivial block for \(H\) is contained either in \(\Omega\) or in \(\Omega^{c}\).

        \noindent Let \(K\) be a non-trivial block. Suppose that \(K\) contains a point of \(\Omega\) and a point of \(\Omega^{c}\). Since every point of \(\Omega^{c}\) is fixed by \(g\), we have $K\cap g(K)\neq\varnothing$, whence $g(K)=K$. Since \(g\) is transitive on \(\Omega\), it follows that $\Omega\subseteq K$. Since $x(\sigma^{r-1}(a_{i,j}))=a_{i,j}=h_i$, we conclude that $x(K)\cap K\neq\varnothing$, whence \(x(K)=K\). Consequently \(K\) contains every cycle of \(x\), whence $K=\{1,2,\ldots, n\}$,  a contradiction.

        \medskip

        \noindent \textbf{Claim 2}: Let \(B\) be a non-trivial block contained in \(\Omega\). Then $B\neq\Omega$.

        \noindent On the contrary, let \(B=\Omega\). Since $\sigma^{r-1}(a_{i,j}),h_i\in\Omega$ and $x(\sigma^{r-1}(a_{i,j}))=a_{i,j}=h_i$, it follows that $x(\Omega)\cap\Omega\neq\varnothing$. Hence $x(\Omega)=\Omega$. If $r>2$, then $x(t_i)=\sigma(a_{i,j})\in \Omega^c$. If $r=2$, since $\lambda\neq (3,2^t)$, there exists a head $h_k$ such that $x(h_k)\in \Omega^c$. We conclude that $x(\Omega)\cap \Omega^c \neq \varnothing$, a contradiction. 

        \medskip

        \noindent \textbf{Claim 3}: Every non-trivial block contained in \(\Omega\) contains at least one head and at least one tail.

        \noindent Since \(g\) is a cycle on \(\Omega\), the claim follows from the structure of blocks for a cyclic permutation.

        \medskip

        \noindent Let \(B\) denote the block containing $\sigma^{r-1}(a_{i,j})$, and let \(C\) denote the block containing \(h_i\). We note that $B\cap C=\varnothing$. Indeed, if $B=C$, from the structure of $g$, it is clear that $B=\Omega$, contradicting the already established claim 2. Since $x(\sigma^{r-1}(a_{i,j}))=h_i$,  we have $h_i\in x(B)$. Hence $x(B)\cap C\neq\varnothing$, and therefore $x(B)=C$.

        \medskip

        \noindent \textbf{Claim 4}: Every head contained in \(B\) is the head of a transposition in the disjoint cycle decomposition of \(x\).

        \noindent Let \(h_k\in B\) where $k\neq i$. Since \(x(B)=C\), we have $x(h_k)\in C$. If the cycle  of $x$ containing \(h_k\) has length greater than \(2\), then $x(h_k)=a_{k,2}$. This contradicts the fact that \(C\subseteq\Omega\). Hence \(h_k\) is the head of a transposition.

        \medskip

        \noindent \textbf{Claim 5}: We have $x(C)=B$. Moreover, every tail contained in \(C\) is the tail of a transposition.

        \noindent Let \(h_j\in B\). By the previous claim, \(h_j\) is the head of a transposition $(h_j,t_j)$.
        Since \(x(B)=C\), $t_j=x(h_j)\in C$. Applying \(x\) once more, we get $h_j=x(t_j)\in x(C)$. Since \(h_j\in B\), we have $x(C)\cap B\neq\varnothing$. Hence $x(C)=B$, as claimed. Now let \(t_j\in C\). Since \(x(C)=B\),
        $x(t_j)\in B$. If \(t_j\) belonged to a cycle of $x$ of length greater than \(2\), then $x(t_j)=h_j$, whence \(h_j\in B\), contradicting the previous claim.

        \medskip 

        \noindent \textbf{Claim 6}: $\lambda_i=r+1$.

        \noindent Since \(h_i\in C\) and \(x(C)=B\), we have $x(h_i)\in B.$ If \(\lambda_i>r+1\), then $x(h_i)=\sigma^{r}(a_{i,j})\in \Omega^c$. This contradicts \(B\subseteq\Omega\), thereby establishing our claim.

        \medskip

        \noindent Assume $r>2$. Note that $x(h_i)=\sigma^{r}(a_{i,j})=t_i\in B$, since $x(C)=B$. Now $x(t_i)=\sigma(a_{i,j})\in \Omega^c$, whence $C\cap \Omega^c\neq \varnothing$, a contradiction. Now assume that $r=2$, whence $\lambda_i=3$. Since $\lambda \neq (3,2^t)$, there exists another cycle of $x$ with length greater than \(2\). Let \(h_k\) be the head of such a cycle and let \(D\) be the block containing \(h_k\). By claim 3, \(D\) contains a tail \(t_s\). Since \(\lambda_k>2\), $x(h_k)=a_{k,2}\in\Omega^{c}$. On the other hand, $x(t_s)$ is a head and therefore belongs to \(\Omega\). Thus \(x(D)\) contains a point of \(\Omega\) and a point of \(\Omega^{c}\), contradicting the first claim. This provides the final contradiction to our assumption, whence $H$ is primitive.
    \end{proof}

    \begin{lemma}\label{primitive_2}
        Let $\lambda=(3,2^t)\vdash n$, where $n\geq 5$. Let $\sigma, \tau, \rho$ be as in Theorem \ref{two_generation_derangements}. Then $H:=\langle \tau \sigma \tau^{-1},\rho \sigma \rho^{-1}\rangle$ is primitive.
    \end{lemma}

    \begin{proof}
        For convenience of the proof, by relabeling if necessary, we let  $\displaystyle \sigma=(1,n,t+2)\prod_{j=1}^{t}(t-j+2,t+j+2)$, $\tau=(1,n)$, $\rho=(1,n,t+2,t+1,t+3,\cdots,2,n-1)$. Set $\displaystyle x:=\tau \sigma\tau^{-1}=(n,1,t+2)\prod_{j=1}^{t}(t-j+2,t+j+2)$ and $\displaystyle y:=\rho \sigma \rho^{-1}=(n,t+2, t+1)\prod_{j=1}^{t}(t-j+1,t+j+2)$. These choices yield $g:=x^{-1}y=(1,2,\ldots, n)$. Since $g$ is an $n$-cycle, a non-trivial block system for $H$ is of the form $\mathcal{B}_d=\{B_1,B_2,\cdots, B_{d}\}$ for some proper divisor $d$ of $n$, where $B_i=\{i+kd \mid 0\leq k\leq n/d -1 \}$. Note that $x(B_1)=B_k$, where $k\equiv  t+2\;(\Mod\;d)$, since $x(1)=t+2$. Clearly $1+d\neq n, t+2$, whence $1+d$ is either $t+j+2$ or $t-j+2$ for some $1\leq j \leq t$. Suppose $1+d=t-j+2$, then $x(1+d)=t+j+2=n-d$. Thus, $n-d\in B_k$, whence $t+2$ is divisible by $d$, a contradiction. Similarly, if $1+d=t+j+2$, then $x(1+d)=t-j+2=n-d$, whence we once again get the same contradiction. This proves the lemma.
    \end{proof}

    \begin{proof}[\textbf{Proof of Theorem \ref{two_generation_derangements}}] 
        Let $x:=\tau \sigma \tau^{-1}$, $y=\rho \sigma \rho^{-1}$. By Lemma \ref{lemma_cycle_length}, Lemma \ref{primitive_1}, Lemma \ref{primitive_2}, $H=\langle x, y \rangle$ is a primitive subgroup of $S_n$ and contains a cycle of length $2l+1$. So if $2\leq 2l+1\leq n-3$, we conclude that $H$ contains $A_n$ by Theorem \ref{Jones_Jordan}. Suppose now that $2l+1\geq n-2$, which implies $l\geq \frac{n-3}{2}$. It is easy to see that we are left with the conjugacy classes indexed by the partitions $(3^u,2^t)$ where $1\leq u\leq 3$, $(4, 2^t)$, $(4,3,2^t)$, and $(5,2^t)$. In all these cases, we exhibit a cycle of desired length in $H$ so that Theorem \ref{Jones_Jordan} can be applied. If $\lambda=(3,2^t)$ or $\lambda=(5,2^t)$, then $x^2\in H$ is a cycle of length 3 or 5, respectively, whence our theorem follows unless $\lambda=(3,2)$ or $\lambda=(5)$ or $\lambda=(5,2)$. In these three remaining cases, \textsf{GAP} (\cite{GAP4}) verifies that $H$ indeed contains $A_5$ and $A_7$, respectively. Now, assume that $\lambda=(4,2^t)$. In this case, one can easily check that $x^2y^2$ is a $5$-cycle. Thus, once again our theorem follows unless $\lambda=(4,2)$, where \textsf{GAP} verifies the theorem. Let $\lambda=(4,3,2^t)$. In this case, $x^4$ is a 3-cycle. Hence the theorem follows. Consider $\lambda=(3,3,2^t)$. In this case, $x^3y^3$ is an $(n-5)$-cycle, whence our result follows unless $\lambda=(3,3)$, which \textsf{GAP} once again verifies in the affirmative. Finally, consider $\lambda=(3,3,3,2^t)$. In this case, $x^3y^3$ is an $(n-8)$-cycle, whence our theorem follows unless $\lambda=(3,3,3)$, which is once again verified by \textsf{GAP} in the affirmative.
    \end{proof}

    \medskip

    \noindent The following corollary is obvious from Theorem \ref{two_generation_derangements}.
    \begin{corollary}\label{several_candidates}
         Let $\lambda=(\lambda_1,\ldots,\lambda_l)\in \mathscr{P}_1(n)$ and $\lambda\neq (2^{n/2})$, where $n\geq 5$. Let $\C_{\lambda}$ be the conjugacy class of \(S_n\) indexed by $\lambda$. Then, for any $x\in \C_{\lambda}$, there exist at least $\displaystyle \sum_{i=1}^{l}\lambda_i(\lambda_i-2)$ many elements $y$ such that $H:=\langle x, y \rangle$ contains $A_n$ or $S_n$.
    \end{corollary}

    \noindent The following corollary establishes $\frac{3}{2}$-conjugate-generation for the conjugacy classes of derangements in $A_n$.

    \begin{corollary}\label{two_generation_split_class}
        Let $C$  be a conjugacy class of derangements in $A_n$, where $n\geq 5$. Then, for any $x\in C$, there exists $y\in C$ such that they together generate $A_n$.
    \end{corollary}
 
    \begin{proof}
         Recall that a $S_n$-conjugacy class $\C_{\lambda}$ indexed by a partition $\lambda$ contained in $A_n$ splits into two $A_n$-classes if and only $\lambda$ has distinct and odd parts. Otherwise, $\C_{\lambda}$ remains a conjugacy class of $A_n$. For the latter one, the corollary follows directly from Theorem \ref{two_generation_derangements}. Suppose now that $\lambda\in \mathscr{P}_1(n)$ has distinct and odd parts. Since $n\geq 5$, clearly, the largest part of $\lambda$ is at least 5.  Let $\C_{\lambda}^{+}$ and $\C_{\lambda}^{-}$ be the split $A_n$-conjugacy. Let $\sigma \in \C_{\lambda}^{\pm}$ be an arbitrary element as in Theorem \ref{two_generation_derangements}. If $n$ is even, $\rho \in S_n\setminus A_n$ and hence $y:=\rho \sigma\rho^{-1}\in \C_{\lambda}^{\mp}$. In this case, we can choose $\tau$ to be a cycle of even length. Thus, $\tau \in S_n\setminus A_n$, whence $x\in \C_{\lambda}^{\mp}$. Thus, we get a pair of elements in $\C_{\lambda}^{\pm}$ that generate $A_n$. If $n$ is odd, $\rho \in A_n$ and hence $y\in \C_{\lambda}^{\pm}$. In this case, we can choose a cycle of odd length, whence $\tau \in A_n$. It follows that $x\in \C_{\lambda}^{\pm}$. Once again, we get a pair of elements from $\C_{\lambda}^{\pm}$ that generate $A_n$. 
    \end{proof}

    \begin{proof}[\textbf{Proof of Theorem \ref{generation_by_conjugate_derangements}}] The proof follows from Theorem \ref{two_generation_derangements} and Corollary \ref{two_generation_split_class}.
    \end{proof}

\subsection*{Acknowledgment} We thank Dr. Sumit Chandra Mishra for several discussions on this topic. The first-named author acknowledges the support of the ANRF-NPDF grant PDF/2025/002961 for the fellowship. The second-named author gratefully acknowledges the financial support received from the New Faculty Seed Grant (NFSG), BITS Pilani, under Grant No. N7/26/805.

\bibliographystyle{abbrv}
\bibliography{References}
\end{document}